\documentclass[11pt,reqno]{amsart}

\usepackage{amsmath, amsthm, amssymb}
\usepackage{amsfonts}
\usepackage{bm,mathrsfs}
\usepackage{enumitem}
\usepackage{esvect}
\usepackage[T1]{fontenc}
\usepackage{lmodern}

\usepackage{hyperref}

\usepackage[ansinew]{inputenc}
\usepackage[dvips]{epsfig}
\usepackage{graphicx}
\usepackage[english]{babel}
\usepackage{mathrsfs}
\usepackage{mathtools}
\DeclarePairedDelimiter{\norm}{\lVert}{\rVert}

\usepackage{thmtools}
\theoremstyle{plain}
\declaretheorem[title=Theorem, parent=section]{theorem}
\declaretheorem[title=Lemma,sibling=theorem]{lemma}
\declaretheorem[title=Proposition,sibling=theorem]{proposition}

\theoremstyle{definition}
\declaretheorem[title=Definition,sibling=theorem]{definition}
\declaretheorem[title=Remark,sibling=theorem]{remark}
\declaretheorem[title=Remark, numbered=no]{remark*}

\declaretheorem[title=Assumption, numbered=no]{assumption*}

\numberwithin{equation}{section}

\usepackage[backgroundcolor=white, bordercolor=blue,
linecolor=blue]{todonotes}

\usepackage{dsfont}
\usepackage{bbm}

\newcommand{\R}{\mathbb{R}}

\DeclareMathOperator*{\osc}{osc}

\newcommand{\average}{{\mathchoice {\kern1ex\vcenter{\hrule height.4pt
width 6pt depth0pt} \kern-9.7pt} {\kern1ex\vcenter{\hrule
height.4pt width 4.3pt depth0pt} \kern-7pt} {} {} }}

\begin{document}
\allowdisplaybreaks
\title{Asymptotically H\"{o}lder gradient regularity for degenerate or singular parabolic normalized $ p$-Laplace equations}

\author{Jiangwen Wang}
\author{Yini Zhang}
\author{Feida Jiang$^*$}

\address{School of Mathematics and Shing-Tung Yau Center of Southeast University, Southeast University, Nanjing 211189, P.R. China}
\email{\url{jiangwen\_wang@seu.edu.cn}}

\address{School of Mathematics, University of Bristol, Bristol BS8 1UG, UK}
\email{\url{ah24858@bristol.ac.uk}}

\address{School of Mathematics and Shing-Tung Yau Center of Southeast University, Southeast University, Nanjing 211189, P.R. China; Shanghai Institute for Mathematics and Interdisciplinary Sciences, Shanghai 200433, P.R. China}
\email{\url{jiangfeida@seu.edu.cn}}

\date{\today}
	\thanks{*corresponding author}

\keywords{degenerate or singular equation; parabolic normalized $ p$-Laplacian; regularity}

\subjclass[2020]{35K55, 35D40}

\allowdisplaybreaks

\begin{abstract}
We study bounded viscosity solutions of the degenerate or singular parabolic equation
\[
 u_t-|Du|^{\gamma}\Delta_p^{\mathrm{N}}u=f
 \quad \text{in } Q_1,
\]
where
\[
 \Delta_p^{\mathrm{N}}u
 :=
 \Delta u+(p-2)
 \left\langle
 D^2u\frac{Du}{|Du|},
 \frac{Du}{|Du|}
 \right\rangle,
 \qquad
 -1<\gamma<\infty,\quad 1<p<\infty,
\]
and $f\in C^0(Q_1)\cap L^\infty(Q_1)$.
Given any prescribed $\alpha\in(0,1)$, we prove that there exists
$\varepsilon=\varepsilon(n,\alpha)>0$ such that
\[
 |p-2|+|\gamma|\leq\varepsilon
 \quad\Longrightarrow\quad
 u\in C_{\mathrm{loc}}^{1+\alpha,\frac{1+\alpha}{2}}(Q_1).
\]
In particular, the spatial gradient H\"older exponent can be chosen arbitrarily close to $ 1 $ as $(p,\gamma)\to(2,0)$.
This extends the result of Andrade and Santos (Calc. Var. Partial Differential Equations \textbf{61}, Paper No. 196, 2022) for $\gamma=0$ to the joint regime $(p,\gamma)\to(2,0)$ and answers affirmatively the question raised in Remark~1.1 therein.   
\end{abstract}

\allowdisplaybreaks

\maketitle

\section{Introduction}

In this paper, we are concerned with the regularity properties of viscosity solutions to degenerate or singular parabolic normalized $ p$-Laplacian equations
 \begin{equation}\label{Main:eq1}
   u_t-|Du|^\gamma\Delta_p^{\mathrm{N}} u=f
 \quad   \text{in }   \quad   Q_1,
 \end{equation}
where $ -1 < \gamma < \infty $, $ 1 < p < \infty $, $ f \in C^{0}(Q_{1}) \cap L^{\infty}(Q_{1}) $, and
\begin{equation}\label{Intro:eq2}
 \Delta_p^{\mathrm{N}} u
 :=\Delta u+(p-2)\left\langle D^2u\frac{Du}{|Du|},\frac{Du}{|Du|}\right\rangle
\end{equation}
is the normalized $p$-Laplacian.

\vspace{1mm}

Equation \eqref{Main:eq1} contains several classical models.  When $ \gamma=0 $ it is the inhomogeneous parabolic normalized $p$-Laplace equation.
When $ \gamma=p-2 $, the second-order part is the usual parabolic $p$-Laplacian written in non-divergence form at noncritical points.  For $\gamma>0$ the diffusion degenerates along the critical set $\{Du=0\}$, whereas for $-1<\gamma<0$ the equation is singular there.

\vspace{1mm}

The interior gradient regularity of \eqref{Main:eq1} is well developed.  For the homogeneous parabolic normalized $p$-Laplacian, noticing that the operator \eqref{Intro:eq2} has a singularity on the set $ \{ Du =0 \} $, which implies that we can not use directly the classic regularity theory of viscosity solutions, as in \cite{CL12, W92}. Jin and Silvestre in \cite{JS17} first established H\"older continuity of the spatial gradient, and Attouchi--Parviainen \cite{AP18} treated the inhomogeneous normalized equation. Shortly after, Imbert--Jin--Silvestre \cite{IJS19} developed a general $C^{1,\alpha}$ theory for singular and degenerate parabolic equations, while Attouchi and Attouchi--Ruosteenoja obtained regularity for the non-divergence model \eqref{Main:eq1} in the degenerate and singular regimes, respectively; see \cite{A20, AR20}.  Variable-exponent normalized parabolic equations were studied by Fang--Zhang \cite{FZ21}, and related borderline and second-order estimates appear in \cite{ABM23, FP23}.  These results provide a positive H\"older exponent depending on the structural parameters, but they do not by themselves identify the behavior of the exponent as the equation approaches the heat equation.

\vspace{1mm}

A different phenomenon was founded by Andrade--Santos \cite{AS22}. For
\begin{equation*}
 u_t-\Delta_p^{\mathrm{N}} u=f  \quad  \text{in }   \quad   Q_1,
\end{equation*}
they proved that for every fixed $\alpha<1$ the solution belongs locally to $C^{1+\alpha, \frac{1+\alpha}{2}}$ provided $p$ is sufficiently close to $2$.  In particular, the gradient becomes asymptotically Lipschitz as $p\to2$. Moreover, in \cite[Remark 1.1]{AS22}, Andrade--Santos explicitly ask the following problem:

\vspace{2mm}

\begin{center}
\ttfamily
`` It is not clear if it is possible to extend \cite[Theorem 1.1]{AS22} to \eqref{Main:eq1} under the assumption that $ |\gamma|+|p-2| \leq \epsilon  $. ''
\end{center}

\vspace{2mm}

The purpose of the present paper is to answer this question affirmatively.

Our main result is the following.

\begin{theorem}
\label{thm:main}
Let $u\in C^{0}(Q_1)$ be a bounded viscosity solution of
\begin{equation*}
 u_t-|Du|^\gamma\Delta_p^{\mathrm{N}} u=f
 \quad   \text{in }    Q_1,
\end{equation*}
where $ 1<p<\infty, -1<\gamma<\infty $, and $f\in C^{0}(Q_1)\cap L^\infty(Q_1)$. Given $\alpha\in(0,1)$, there exists
$
 \varepsilon =\varepsilon(\alpha, n)>0 $ such that if
\[
|p-2|+|\gamma|\le\varepsilon,
\]
then
\[
 u\in C^{1+\alpha, \frac{1+\alpha}{2}}(Q_{1/2}).
\]
Moreover,
\[
 ||u||_{C^{1+\alpha,\frac{1+\alpha}{2}}(Q_{1/2})}
 \le \mathrm{C}   
\]
where $ \mathrm{C} $ depends only on $n, \gamma $, $\alpha$, $ ||u||_{L^{\infty}(Q_{1})} $ and $ ||f||_{L^{\infty}(Q_{1})} $. Consequently the spatial gradient H\"older exponent may be chosen arbitrarily close to $ 1 $ when $(p,\gamma)$ is sufficiently close to $(2,0)$.
\end{theorem}

\begin{remark}\label{rem:published-question}
Theorem~\ref{thm:main} addresses the question stated in \cite[Remark~1.1]{AS22}. The point is not merely that solutions of \eqref{Main:eq1} are $C^{1,\alpha_0}$ for some structural exponent $\alpha_0>0$; rather, the exponent is prescribed in advance and can be arbitrarily close to $1 $, at the cost of shrinking a joint neighborhood of $ (p,\gamma)=(2,0) $.
\end{remark}

\vspace{2mm}

The result is consistent with recent developments but is of a different nature. Boundary $C^{1,\alpha}$ estimates for general parabolic $p$-Laplace type equations were developed in \cite{LLYZ25}; in that setting the near-Lipschitz improvement for $p$ close to $2$ is singled out in the case $ \gamma=0 $. A related asymptotic regularity phenomenon was established by Pimentel--Rampasso--Santos \cite{PRS20} for the elliptic $ p$-Poisson equation: for every prescribed $ \alpha \in (0,1) $, solutions are locally $ C^{1, \alpha} $ when $ p $ is sufficiently close to $ 2 $, via a perturbative argument based on harmonic approximation and a gradient-dependent separation of critical and uniformly elliptic regimes. For more related results, see \cite{ASU26, AR18, WYJ25, AS26, W26}. The present problem requires an additional control of the intrinsic time geometry and of the transition between critical and noncritical gradient regimes.  To the best of our knowledge, Theorem~\ref{thm:main} is the first joint $(p,\gamma)\to(2,0)$ asymptotically Lipschitz gradient result for the inhomogeneous parabolic equation \eqref{Main:eq1}.

\subsection{Ideas of proof of Theorem~\ref{thm:main}}

The ideas is inspired in part by the elliptic perturbative approach of \cite{PRS20} and \cite{CC95}. In the present parabolic setting, however, the argument has to be combined with an intrinsic scaling and a viscosity analysis of the zero-gradient set.

The argument starts from a compactness principle (Lemma~\ref{Section 3:lem1}) showing that normalized deviations from affine functions are close to caloric functions when $ |p-2| + |\gamma| + ||f||_{\infty} $ is sufficiently small, see Lemma~\ref{Section 3:lem2}. This yields a one-step improvement of flatness (Lemma~\ref{lem:onestep}), which is then iterated on the intrinsic cylinders associated with the gradient scale. The iteration continues as long as the normalized affine slope remains bounded. If this happens at every scale, one obtains the desired first-order expansion directly, see Section~\ref{Section 5.1}. Otherwise, at the first stopping scale the affine slope dominates the remaining oscillation, and the equation enters a noncritical regime where uniformly parabolic
estimates become available, see Lemma~\ref{lem:smoothimproved}.  The estimates from the two regimes are then patched together, and a Campanato comparison of tangent planes (Proposition~\ref{prop:campanato-consequences}) gives the H\"older continuity of the gradient.

\vspace{1mm}

Several technical obstacles arise in implementing this program. In particular:

\begin{enumerate}[label=\roman*)]

\item The first is the loss of the standard
parabolic scaling when $ \gamma \neq 0 $.  If the gradient error at scale $ r $  is of order $ \beta $, then the natural
time scale is $ r^{2-\beta \gamma}$. So the flatness iteration has to be performed in an intrinsic geometry rather than on the usual parabolic cylinders in \cite{AS22};

\vspace{3mm}

\item A second difficulty is to show that the limit obtained by compactness
solves the heat equation. Although $|\xi|^\gamma\to1$ as $\gamma\to0$
for each fixed $\xi\ne0$, the convergence is not uniform near $\xi=0$.
To see where this matters, suppose that $w_j\to w_\infty$ locally
uniformly, $(p_j,\gamma_j)\to(2,0)$, and the gradient shifts satisfy
$q_j\to q_\infty$. Let $\varphi\in C^{2,1}$ touch $w_\infty$ strictly
from above at $(0,0)$. Then there exist constants $c_j$ and points
$(x_j,t_j)\to(0,0)$ such that $\varphi+c_j$ touches $w_j$ from above
at $(x_j,t_j)$. The corresponding total gradients satisfy
\[
\xi_j:=D\varphi(x_j,t_j)+q_j
\longrightarrow D\varphi(0,0)+q_\infty.
\]
If $ D\varphi(0,0)+q_\infty \neq 0 $, then $|\xi_j|^{\gamma_j}\to1$,
and we may pass to the limit in the viscosity inequality. If it vanishes,
however, $|\xi_j|^{\gamma_j}$ need not converge to $1$, even when
$\xi_j\ne0$ for every $j$.

At such critical contacts, we argue by contradiction. We construct
comparison functions whose spatial parts are homogeneous of degree
$\kappa_j= \frac{2+\gamma_j}{1+\gamma_j} $. Since
\[
\gamma_j(\kappa_j-1)+\kappa_j-2=0,
\]
the powers of $|x|$ contributed by the gradient factor and the Hessian
cancel in the diffusion term. Estimates for the rescaled gradients and
Hessians then show that the diffusion terms converge to the Laplacian
of the limiting quadratic polynomial, uniformly for $x\ne0$. After
verifying that the comparison functions are viscosity super-solutions
also at $x=0$, we apply the comparison principle to rule out
$\varphi_t(0,0)>\Delta\varphi(0,0)$, see Lemma \ref{Section 3:lem2} below;

\vspace{3mm}

\item Finally, in the stopping regime one needs estimates which are uniform as $ \gamma \to 0 $ from both the
degenerate ($ \gamma > 0 $) and singular ($ -1 < \gamma < 0$) sides. A uniform Lipschitz estimate for the normalized deviation places all relevant test gradients in a fixed annulus away from zero. On this annulus the coefficient
matrix is uniformly close to $ \textbf{I}\mathrm{d_{n}} $; after a harmless global extension, the equation becomes a uniformly parabolic perturbation of the heat equation and the standard improvement-of-flatness argument applies. In the singular case $ -1 < \gamma < 0 $, the intrinsic time exponent is slightly weaker than the standard one. This loss is compensated by carrying out the iteration with an auxiliary exponent $ \beta > \alpha $ and then choosing $ |\gamma|$ sufficiently small.
\end{enumerate}

\subsection{Structure of the paper} The paper is organized as follows. In Section~\ref{Section 2}, we introduce the viscosity
framework. In Section~\ref{Section 3}, we are devoted to establish the uniform compactness and heat-approximation results.
Section~\ref{Section 4} proves the one-step intrinsic improvement of flatness. Section~\ref{Section 5} develops the critical iteration and treats the stopping situation, including the noncritical Lipschitz and smooth-regime estimates.
In Section~\ref{Section 6}, we combine the two alternatives to obtain pointwise affine approximations and their campanato consequences. Finally, Section~\ref{Section 7} completes the proof of Theorem~\ref{thm:main}.

\subsection{Notation}
We summarize below the basic notation used throughout the paper.
\begin{itemize}

\item For $z_0=(x_0,t_0)\in\R^{n+1}$ and $r>0$, set
\[
 Q_r(z_0):=B_r(x_0)\times(t_0-r^2,t_0],
 \qquad Q_r:=Q_r(0,0).
\]

\item For $\lambda>0$ we also use the intrinsic cylinder
\begin{equation*}
 Q_r^\lambda(z_0)
 :=B_r(x_0)\times(t_0-r^2\lambda^{-\gamma},t_0],
 \qquad Q_r^\lambda:=Q_r^\lambda(0,0).
\end{equation*}

\item $ \text{Sym}(n) $ denotes the space of all $ n \times n $ symmetric matrices in $ \mathbb{R}^{n} $.

\item For $0<\sigma<1$ and a parabolic cylinder $Q$, define
\[
 [v]_{C^{\sigma,\frac{\sigma}{2}}(Q)}
 :=\sup_{(x,t)\ne(y,s)\in Q}
 \frac{|v(x,t)-v(y,s)|}{|x-y|^\sigma+|t-s|^{\frac{\sigma}{2}}}.
\]
We use the standard norm
\begin{align*}
 \norm{u}_{C^{1+\alpha,\frac{1+\alpha}{2}}(Q)}
 :=&\ \norm{u}_{L^\infty(Q)}+\norm{Du}_{L^\infty(Q)}
 +[Du]_{C^{\alpha,\frac{\alpha}{2}}(Q)}\\
 &+\sup_{(x,t), (x,s) \in Q, t \neq s}
 \frac{|u(x,t)-u(x,s)|}{|t-s|^{\frac{1+\alpha}{2}}},
\end{align*}
with the suprema restricted to points for which the relevant pairs belong to $ Q $.

\item $ \textbf{I}\mathrm{d_{n}} $ denotes the $ n \times n $ identity matrix.

\item $ \mathrm{C} $ shall denote a generic positive constant which may vary in different inequalities.
\end{itemize}

\vspace{2mm}

\section{Preliminaries}\label{Section 2}

In this section, we recall the viscosity framework for \eqref{Main:eq1} and record its basic scaling properties. These preliminaries will be used repeatedly in the compactness and improvement-of-flatness arguments below.

\subsection{Viscosity solution} We adopt the admissible-test formulation of Imbert-Jin-Silvestre \cite[Definition~5.1]{IJS19}, with the continuous right-hand side $f$ in~\eqref{Main:eq1}.

\begin{definition}\label{def:viscosity}
Let $1<p<\infty$, $\gamma>-1$, and
$f\in C^{0}(Q_1)\cap L^\infty(Q_1)$. Denote by $\mathcal A_{p,\gamma}$
the class of admissible $C^2$ test functions introduced in
\cite[Section~5]{IJS19} for the operator
$|Du|^\gamma\Delta_p^N u$.

A locally bounded upper semicontinuous function $u$ in $Q_1$
is a viscosity sub-solution of \eqref{Main:eq1} if, whenever
$\varphi\in\mathcal A_{p,\gamma}$ touches $u$ from above at
$z_0=(x_0,t_0)\in Q_1$, one has
\begin{equation}\label{Section2:eq1}
\begin{cases}
  \varphi_t(z_0)
  -|D\varphi(z_0)|^\gamma\Delta_p^N\varphi(z_0)
  \leq f(z_0),
  & \text{if } D\varphi(z_0)\neq0, \\[4pt]
  \varphi_t(z_0)\leq f(z_0),
  & \text{if } D\varphi(z_0)=0.
\end{cases}
\end{equation}
A locally bounded lower semicontinuous function $u$ is a viscosity
super-solution if every admissible test touching $u$ from below
satisfies the reversed inequalities in \eqref{Section2:eq1}.
A continuous function is a viscosity solution if it is both a
viscosity sub-solution and a viscosity super-solution.
\end{definition}

\begin{remark}
When $\gamma=0$, we use the standard viscosity interpretation of the normalized $p$-Laplacian. At a touching point with $D\varphi \ne 0$ the inequality is \eqref{Section2:eq1} with $\gamma=0$.  At $ D\varphi =0 $, the subsolution condition is
\[
 \varphi_t-\Delta \varphi-(p-2)\lambda_{\max}(D^2\varphi)\le f
 \quad\text{if }p\ge2,
\]
and
\[
\varphi_t- \Delta\varphi -(p-2)\lambda_{\min}(D^{2}\varphi)\le f
 \quad\text{if }1<p<2.
\]
For supersolutions, $\lambda_{\max}$ and $\lambda_{\min}$ are interchanged.  This is exactly the semicontinuous-envelope convention used for the parabolic normalized $p$-Laplacian in \cite{AS22, AP18}.
\end{remark}

\subsection{Scaling properties}\label{Section2.2}

In what follows, the basic scaling of \eqref{Main:eq1} will be used repeatedly. More precisely, given $ (x_{0}, t_{0}) \in Q' \Subset Q_{1} $, for $r,A, \epsilon, \eta>0$, let
\begin{equation*}\label{eq:basic-scaling-v}
 v(x,t):=\frac{u(x_0+rx,t_0+\tau t)}{A},
 \quad
 \tau:=r^{2+\gamma}A^{-\gamma}.
\end{equation*}
A direct computation yields that $ v $ solves, in the viscosity sense, the following equation,
\begin{equation*}\label{eq:basic-scaling-rhs}
 v_t-|Dv|^\gamma\Delta_p^{\mathrm{N}} v
 = \widetilde{f}(x,t)  \quad    \text{in }  Q_{1},
\end{equation*}
where
\begin{equation*}
  \widetilde{f}(x,t):= \frac{r^{2+\gamma}}{A^{1+\gamma}}
 f(x_0+rx,t_0+\tau t).
\end{equation*}
To ensure $ ||v||_{L^{\infty}(Q_{1})} \leq \eta $ and $ ||\widetilde{f}||_{L^{\infty}(Q_{1})} \leq \epsilon $, we choose the spatial scale $ r > 0 $ sufficiently small and then set
\begin{equation*}
  A:= 1+ \frac{1}{\eta}||u||_{L^{\infty}(Q_{1})} + ||f||_{L^{\infty}(Q_{1})}^{\frac{1}{1+\gamma}}.
\end{equation*}

\vspace{2mm}

\section{Uniform compactness}\label{Section 3}

In this section, we begin with the shifted deviation equation
\begin{equation}\label{eq:shifted}
 w_t-|Dw+q|^\gamma
 \left[
 \Delta w+(p-2)
 \left\langle
 D^2w\frac{Dw+q}{|Dw+q|},
 \frac{Dw+q}{|Dw+q|}
 \right\rangle
 \right]=f,
\end{equation}
with a fixed vector $ q \in \R^n $. At points where $Dw+q=0$, this is understood through the viscosity formulation inherited by $w+q\cdot x $. The purpose of this section is to obtain compactness uniformly with respect to the affine shift $ q $ in a fixed bounded set and with respect to $ (p,\gamma) $ near $ (2,0) $.

Now we provide a compactness result for viscosity solutions of equation~\eqref{eq:shifted}.

\begin{lemma}\label{Section 3:lem1}
Fix $M>0$ and $\gamma_0\in(0,\frac{1}{4}]$. There exist
$\theta\in(0,1)$ and $ \mathrm{C} > 0$, depending only on $n,M,\gamma_0$, such that
the following holds for
\[
 |p-2|\leq\frac12,\qquad |\gamma|\leq\gamma_0,
 \qquad |q|\leq M.
\]
If $w$ solves \eqref{eq:shifted} in $Q_1$ and
\[
 \norm{w}_{L^{\infty}(Q_1)}+\norm{f}_{L^{\infty}(Q_1)}\leq1,
\]
then
\[
 |w(x,t)-w(y,s)|
 \leq \mathrm{C}  \big(|x-y|^{\theta}+|t-s|^{\theta/2}\big)
\]
for $(x,t),(y,s)\in Q_{3/4}$.
\end{lemma}

\begin{proof}
Set $ v(x,t):= \omega(x,t) + q\cdot x    $, then it can be seen that $ v $ satisfies \eqref{Main:eq1} with right-hand side $ f $ and \[
 \norm{v}_{L^{\infty}(Q_1)}\leq1+M.
\]

Applying the interior Lipschitz estimate and the time-H\"{o}lder estimate established by \cite[Lemmas 3.1 and 3.2]{A20}, then we have
\begin{equation}\label{Section 3:eq1}
  |v(x,t)- v(y,t)| \leq \mathrm{C}_{L}|x-y|,
\end{equation}
where $ \mathrm{C}_{L} := \mathrm{C}_{L}(n, M, \gamma_{0}) $, and
\begin{equation*}
  |v(x,t)-v(x,s)| \leq C_{T}|t-s|^{\nu_{\gamma}}
\end{equation*}
with $ \mathrm{C}_{T} = \mathrm{C}_{T}(n, M, \gamma_{0})$ and
\begin{equation*}
  \nu_\gamma:=
 \min\left\{\frac12,\frac1{2+\gamma}\right\}.
\end{equation*}
Noticing that $ \nu_\gamma \geq \frac{4}{9} $ since $ |\gamma| \leq \frac{1}{4} $. Subtracting the fixed plane $q\cdot x$ from \eqref{Section 3:eq1} preserves a common spatial Lipschitz modulus for $w$, while the plane is independent of time.  We may therefore take, for example, $   \theta:=\frac12 $, after weakening the time exponent, and obtain the asserted estimate on $ Q_{3/4} $.
\end{proof}

The previous lemma provides the compactness for the following heat approximation lemma:

\begin{lemma}\label{Section 3:lem2}
Fix $M>0$ and $\delta>0$. There exists $ \varepsilon=\varepsilon(n,M,\delta)>0 $ such that if
\[
 |p-2|+|\gamma|+\norm{f}_{L^{\infty}(Q_1)}\leq\varepsilon,
 \qquad |q|\leq M,
\]
and $w$ is a viscosity solution of \eqref{eq:shifted} with
$ \norm{w}_{L^{\infty}(Q_1)}\leq1 $, then there exists a caloric function $h$ in $Q_{2/3}$ such that
\[
 \norm{w-h}_{L^{\infty}(Q_{1/2})}\leq\delta.
\]
\end{lemma}

\begin{proof}
Argue by contradiction, suppose that there exist sequences
\[
 p_j\to2,\qquad \gamma_j\to0,\qquad
 \norm{f_j}_{L^{\infty}(Q_1)}\to0,
 \qquad |q_j|\leq M,
\]
and solutions $w_j$ of \eqref{eq:shifted} such that
$  \norm{w_j}_{L^{\infty}(Q_1)}\leq 1 $. However, for every caloric function $ h $ in $ Q_{1/2} $, and some $ \delta_0 > 0 $, independent of $ j $, we have that
\begin{equation}\label{Section 3:eq2}
  ||\omega_{j}-h||_{L^{\infty}(Q_{1/2})} > \delta_0.
\end{equation}

After passing to a subsequence, $q_j\to q_\infty$.  Lemma~\ref{Section 3:lem1} and Arzel\`a--Ascoli theorem give
\[
 w_j\longrightarrow w_\infty
 \quad\text{locally uniformly in }Q_{3/4}.
\]
We claim that
\begin{equation}\label{Section 3:eq3}
 (w_\infty)_t-\Delta w_\infty=0
 \quad    \text{in }  Q_{3/4}
\end{equation}
in the viscosity sense.

For this purpose, we only prove the case of sub-solutions and the proof for super-solutions is similar. Let $ \varphi \in C^{2,1} $ touch $ \omega_{\infty} $ from above at $ (x_0, t_0) $. For simplicity, we assume that $ (x_0, t_0)= (0,0) $ and $ \omega_{\infty}(0,0)= \varphi(0,0)=0 $. Now we want to show that
\begin{equation}\label{Lemma3.2:eq130}
  \varphi_{t}(0,0)- \Delta \varphi(0,0) \leq 0.
\end{equation}
To prove \eqref{Lemma3.2:eq130}, we consider the following two cases:

\medskip
\noindent {\bf Case 1}. $ D\varphi(0,0) + q_{\infty} \neq 0 $. Choose $r>0$ sufficiently small that $w_j\to w_\infty$
uniformly on $\overline{Q_r}$ and $w_\infty-\varphi$ attains
its unique maximum there at $(0,0)$. For each $j$, choose
$(x_j,t_j)\in\overline{Q_r}$ such that
\[
(w_j-\varphi)(x_j,t_j)
=\max_{\overline{Q_r}}(w_j-\varphi).
\]
By maximality, we have
\[
(w_j-\varphi)(x_j,t_j)\ge (w_j-\varphi)(0,0).
\]
Passing to the limit along any convergent subsequence and using
uniform convergence, we find that every accumulation point of
$(x_j,t_j)$ maximizes $w_\infty-\varphi$ on $\overline{Q_r}$.
Since this maximum is unique, $(x_j,t_j)\to(0,0)$.
Consequently, for all sufficiently large $j$,
\[
x_j\in B_r,\quad -r^2<t_j\le0,
\]
so these points lie neither on the initial face nor on the
lateral boundary of $Q_r$.

We define
\[
 \varphi_j(x,t)
   :=\varphi(x,t)+w_j(x_j,t_j)-\varphi(x_j,t_j).
\]
By maximality, it is readily seen that
\begin{equation*}
  \omega_{j} - \varphi_{j}  \leq 0 \quad  \text{on} \quad \overline{Q_r},   \quad   \omega_{j}(x_j, t_j) = \varphi_{j}(x_j, t_j),
\end{equation*}
hence $ \varphi_{j}  $ touches $ \omega_{j} $ from above at $ (x_j, t_j) $. Moreover, we have that
\begin{equation*}
  D\varphi_{j}(x_j, t_j) +q_{j} \longrightarrow D\varphi(0,0)+q_{\infty} \neq 0,
\end{equation*}
which implies that $ |D\varphi_{j}(x_j, t_j) +q_{j}| $ is nonzero for all sufficiently large $ j $. Now we may use $ \varphi_{j} $ in the viscosity sub-solution inequality:
\begin{equation}\label{Lemma3.2:eq131}
 (\varphi_j)_t-|D\varphi_j+q_j|^{\gamma_j}
 \left[\Delta\varphi_j+(p_j-2)
 \frac{\langle D^2\varphi_j(D\varphi_j+q_j),D\varphi_j+q_j\rangle}
 {|D\varphi_j+q_j|^2}\right]\le f_j(x_j,t_j),
 \end{equation}
where all derivatives are evaluated at $ (x_j, t_j) $.

Noting that as $ \gamma_{j} \to 0 $, we have
\begin{equation}\label{Lemma3.2:eq132}
  |D\varphi_{j}(x_j, t_j) +q_j|^{\gamma_{j}} = \exp\!\left(\gamma_j\log|D\varphi_j(x_j,t_j)+q_j|\right)
       \longrightarrow 1, \quad   \text{as }  j \to \infty.
\end{equation}
In addition, at the same points $ (x_j, t_j) $, one easily sees that
\begin{equation}\label{Lemma3.2:eq133}
  \left|(p_j-2)
 \frac{\langle D^2\varphi_j(D\varphi_j+q_j),D\varphi_j+q_j\rangle}
 {|D\varphi_j+q_j|^2}\right|
 \le |p_j-2|\,\|D^2\varphi(x_j,t_j)\|
       \longrightarrow0, \quad   \text{as }  j \to \infty.
\end{equation}

We combine \eqref{Lemma3.2:eq132}--\eqref{Lemma3.2:eq133} and $ \norm{f_j}_{L^{\infty}(Q_1)} \to 0 $, and pass to the limit in \eqref{Lemma3.2:eq131}, it follows that \eqref{Lemma3.2:eq130} holds.

\medskip
\noindent {\bf Case 2}. $D\varphi(0,0)+q_{\infty}=0$.
Suppose, for contradiction, that~\eqref{Lemma3.2:eq130}
fails, namely,
\begin{equation*}
\varphi_t(0,0)-\Delta\varphi(0,0)>0.
\end{equation*}

Set
\begin{equation*}
a:=\varphi_t(0,0)-\delta,
\quad
\mathrm{M}:=D^2\varphi(0,0)
             +2\delta\textbf{I}\mathrm{d_n}.
\end{equation*}
Then a straightforward calculation yields
\begin{equation*}\label{eq:heat-positive-margin}
\begin{aligned}
a-\mathrm{tr}(\mathrm{M})
&=\varphi_t(0,0)-\Delta\varphi(0,0)-(1+2n)\delta\\
&>\frac12
  \bigl[\varphi_t(0,0)-\Delta\varphi(0,0)\bigr]
>0,
\end{aligned}
\end{equation*}
provided we choose
\begin{equation*}
0<\delta<
\frac{\varphi_t(0,0)-\Delta\varphi(0,0)}
     {2(1+2n)}.
\end{equation*}
We may additionally choose $\delta$ so that $\mathrm{M}$
is non-singular. In fact, if
$\lambda_1,\lambda_2,\ldots,\lambda_n$ are the eigenvalues
of $D^2\varphi(0,0)$, then the eigenvalues of $\mathrm{M}$
are $\lambda_i+2\delta$, $i=1,\ldots,n$.
Therefore, it suffices to avoid the finitely many values
\begin{equation*}
\delta=-\frac{\lambda_i}{2},
\qquad i=1,\ldots,n,
\end{equation*}
which guarantees that
\begin{equation*}
\mathrm{det}(\mathrm{M})
=\prod_{i=1}^{n}(\lambda_i+2\delta)\neq0.
\end{equation*}

Since $\varphi$ touches $w_{\infty}$ from above at $(0,0)$,
we have $w_{\infty}\leq\varphi$ near $(0,0)$.
Using $w_{\infty}(0,0)=\varphi(0,0)=0$ and
$D\varphi(0,0)+q_{\infty}=0$, Taylor's formula yields
\begin{equation}\label{eq:heat-upper-expansion}
\begin{aligned}
w_{\infty}(x,t)+q_{\infty}\cdot x
&\leq\varphi(x,t)+q_{\infty}\cdot x\\
&=\varphi_t(0,0)t
  +\frac12\langle D^2\varphi(0,0)x,x\rangle
  +o(|x|^2+|t|).
\end{aligned}
\end{equation}

Set
\begin{equation*}
\Psi(x,t):=at+\frac12\langle\mathrm{M}x,x\rangle.
\end{equation*}
Then, by the definitions of $a$ and $\mathrm{M}$,
for $t\leq0$ we obtain
\begin{equation*}
\begin{aligned}
&\Psi(x,t)-\varphi_t(0,0)t
-\frac12\langle D^2\varphi(0,0)x,x\rangle\\
&\quad=
\bigl[a-\varphi_t(0,0)\bigr]t
+\frac12
 \langle[\mathrm{M}-D^2\varphi(0,0)]x,x\rangle\\
&\quad=-\delta t+\delta|x|^2
=\delta(|x|^2+|t|).
\end{aligned}
\end{equation*}
Since $\delta>0$ is fixed, we may choose $r\in(0,1)$
sufficiently small that the touching inequality holds on
$\overline{B_r}\times[-r^2,0]$ and the absolute value
of the remainder in~\eqref{eq:heat-upper-expansion}
is bounded by $\frac{\delta}{2}(|x|^2+|t|)$ there.
Combining this bound with~\eqref{eq:heat-upper-expansion}
and the preceding identity gives
\begin{equation}\label{eq:heat-cylinder-gap}
w_{\infty}(x,t)+q_{\infty}\cdot x-\Psi(x,t)
\leq-\frac{\delta}{2}(|x|^2+|t|)
\quad\text{on }\overline{B_r}\times[-r^2,0].
\end{equation}
In particular, $|x|^2+|t|\geq r^2$ on the initial
and lateral boundary $\partial_{\mathrm{par}}Q_r$.
Therefore, it follows from~\eqref{eq:heat-cylinder-gap}
that
\begin{equation}\label{eq:heat-boundary-gap}
w_{\infty}(x,t)+q_{\infty}\cdot x-\Psi(x,t)
\leq-\frac{\delta r^2}{2}
\quad\text{on }\partial_{\mathrm{par}}Q_r,
\end{equation}
with $\Psi(0,0)=w_{\infty}(0,0)=0$.

We now derive a contradiction in three steps.

\medskip
\noindent {\bf Step 1. Construction of the comparison functions.} Set
\begin{equation*}\label{eq:heat-comparison-profiles}
\left\{
\begin{aligned}
\kappa_j
&:=\frac{2+\gamma_j}{1+\gamma_j},\\[4pt]
P_j(x)
&:=\frac12|x|^{\kappa_j-2}
       \langle\mathrm{M}x,x\rangle
       \quad\text{for }x\neq0,
       \quad P_j(0):=0,\\[4pt]
\Psi_j(x,t)
&:=at+P_j(x).
\end{aligned}
\right.
\end{equation*}
Since $\gamma_j>-1$ and $\gamma_j\to0$, we have
$\kappa_j>1$, $\kappa_j\to2$, and
\begin{equation}\label{eq:heat-exponent-cancellation}
\gamma_j(\kappa_j-1)+\kappa_j-2=0.
\end{equation}
For $x\neq0$, a routine calculation reveals that
\begin{equation}\label{eq:heat-profile-gradient}
DP_j(x)
=
|x|^{\kappa_j-2}\mathrm{M}x
+\frac{\kappa_j-2}{2}|x|^{\kappa_j-4}
       \langle\mathrm{M}x,x\rangle x,
\end{equation}
and
\begin{equation}\label{eq:heat-profile-hessian}
\begin{aligned}
D^2P_j(x)
={}&|x|^{\kappa_j-2}\mathrm{M} +(\kappa_j-2)|x|^{\kappa_j-4}
  \bigl[(\mathrm{M}x)\otimes x
        +x\otimes(\mathrm{M}x)\bigr]\\
& \quad  +\frac{\kappa_j-2}{2}|x|^{\kappa_j-4}
  \langle\mathrm{M}x,x\rangle
  \textbf{I}\mathrm{d_n}\\
&\qquad +\frac{(\kappa_j-2)(\kappa_j-4)}{2}
  |x|^{\kappa_j-6}
  \langle\mathrm{M}x,x\rangle x\otimes x.
\end{aligned}
\end{equation}
Using $|\mathrm{M}x|\leq\|\mathrm{M}\||x|$ and
$|\langle\mathrm{M}x,x\rangle|
 \leq\|\mathrm{M}\||x|^2$, we obtain from
\eqref{eq:heat-profile-gradient} that
\begin{equation*}
\left|
|x|^{1-\kappa_j}DP_j(x)
-\mathrm{M}\frac{x}{|x|}
\right|
\leq\frac{|\kappa_j-2|}{2}\|\mathrm{M}\|.
\end{equation*}
Analogously, since $\frac{3}{2}\leq\kappa_j\leq \frac{5}{2}$ for all
sufficiently large $j$, \eqref{eq:heat-profile-hessian}
gives
\begin{equation}\label{eq:heat-rescaled-hessian}
\left\|
|x|^{2-\kappa_j}D^2P_j(x)-\mathrm{M}
\right\|
\leq4|\kappa_j-2|\|\mathrm{M}\|.
\end{equation}

Noticing that $\mathrm{M}$ is non-singular, we have
$\min_{|e|=1}|\mathrm{M}e|>0$.
Consequently, there exist constants
$\mathrm{c},\mathrm{C}>0$, depending only on $\mathrm{M}$,
such that
\begin{equation}\label{eq:heat-rescaled-gradient}
\mathrm{c}
\leq |x|^{1-\kappa_j}|DP_j(x)|
\leq\mathrm{C}
\quad\text{for all }x\neq0
\end{equation}
and all sufficiently large $j$.
In particular, $DP_j(x)\neq0$ whenever $x\neq0$.

By \eqref{eq:heat-exponent-cancellation}, we obtain
\begin{equation}\label{eq:heat-diffusion-factorization}
\begin{aligned}
&|DP_j|^{\gamma_j}
\left[
\Delta P_j+(p_j-2)
\frac{\langle D^2P_jDP_j,DP_j\rangle}{|DP_j|^2}
\right]\\
&\quad=
\bigl(|x|^{1-\kappa_j}|DP_j|\bigr)^{\gamma_j}
\left[
|x|^{2-\kappa_j}\Delta P_j
+(p_j-2)
\frac{
\langle |x|^{2-\kappa_j}D^2P_jDP_j,DP_j\rangle
}{
|DP_j|^2
}
\right].
\end{aligned}
\end{equation}
Indeed, the radial factor arising on the right-hand side
is $|x|^{-\gamma_j(\kappa_j-1)+2-\kappa_j}=1$. It follows from \eqref{eq:heat-rescaled-gradient} that
$\log\bigl(|x|^{1-\kappa_j}|DP_j(x)|\bigr)$
is uniformly bounded. Since $\gamma_j\to0$, this implies
\begin{equation*}
\bigl(|x|^{1-\kappa_j}|DP_j(x)|\bigr)^{\gamma_j}
\longrightarrow1
\quad\text{uniformly for }x\neq0.
\end{equation*}
Moreover, \eqref{eq:heat-rescaled-hessian} yields
\begin{equation*}
|x|^{2-\kappa_j}\Delta P_j(x)
\longrightarrow\mathrm{tr}(\mathrm{M})
\quad\text{uniformly for }x\neq0.
\end{equation*}
The remaining term satisfies
\begin{equation*}
\begin{aligned}
&\left|
(p_j-2)
\frac{
\langle |x|^{2-\kappa_j}D^2P_jDP_j,DP_j\rangle
}{
|DP_j|^2
}
\right|\\
&\quad\leq
|p_j-2|
\left(
\|\mathrm{M}\|
+4|\kappa_j-2|\|\mathrm{M}\|
\right)
\longrightarrow0.
\end{aligned}
\end{equation*}
Therefore, the expression in
\eqref{eq:heat-diffusion-factorization} converges to
$\mathrm{tr}(\mathrm{M})$ uniformly for $x\neq0$.
Since $a-\mathrm{tr}(\mathrm{M})>0$, for all sufficiently
large $j$, we have
\begin{equation}\label{eq:heat-strict-supersolution}
\begin{aligned}
&(\Psi_j)_t
-|D\Psi_j|^{\gamma_j}
\left[
\Delta\Psi_j+(p_j-2)
\frac{
\langle D^2\Psi_jD\Psi_j,D\Psi_j\rangle
}{
|D\Psi_j|^2
}
\right]\\
&\qquad\geq
\frac{a-\mathrm{tr}(\mathrm{M})}{2}>0
\quad\text{for }x\neq0.
\end{aligned}
\end{equation}

Moreover, using the mean value theorem, it is readily seen that
\begin{equation*}
|s^{\kappa_j}-s^2|
\leq|\kappa_j-2|s^{3/2}|\log s|
\quad\text{for }0<s\leq r<1
\end{equation*}
and all sufficiently large $j$. Hence
\begin{equation*}
\begin{aligned}
&\sup_{x\in\overline{B_r}}
\left|
P_j(x)-\frac12\langle\mathrm{M}x,x\rangle
\right|  \\
&\qquad\leq
\frac{\|\mathrm{M}\|}{2}|\kappa_j-2|
\sup_{0<s\leq r}s^{3/2}|\log s|
\longrightarrow0.
\end{aligned}
\end{equation*}
Consequently,
\begin{equation}\label{eq:heat-comparison-convergence}
\Psi_j\longrightarrow\Psi
\quad\text{uniformly on }
\overline{B_r}\times[-r^2,0].
\end{equation}

\medskip
\noindent {\bf Step 2. Verification for $\Psi_j$ at $x=0$.} By \eqref{eq:heat-strict-supersolution}, it remains to
verify the supersolution property at points of the form
$(0,t_*)$. Since $\Psi_j$ is defined for all times,
this verification can be carried out on an open time
interval containing $[-r^2,0]$.

Let $\zeta$ touch $\Psi_j$ from below at $(0,t_*)$.
For $\gamma_j<0$, we take $\zeta$ to be admissible.
Since $\kappa_j>1$ and $\Psi_j(0,t)=at$, we have
\begin{equation*}
D\zeta(0,t_*)=0,
\quad
\zeta_t(0,t_*)=a.
\end{equation*}

We consider the following three possibilities.

\smallskip
\noindent
(i) $\gamma_j>0$.
In this case, $1<\kappa_j<2$.
Suppose first that $\mathrm{M}$ has a negative eigenvalue
$-\lambda$, with $\lambda>0$, and let $ e $ be a
corresponding unit eigenvector. Then
\begin{equation*}
P_j(s e)=-\frac{\lambda}{2}|s|^{\kappa_j}.
\end{equation*}
Since $D\zeta(0,t_*)=0$, Taylor's formula and the lower
touching inequality give, for some $ \mathrm{C} >0$ and all
sufficiently small $s$,
\begin{equation*}
- \mathrm{C}s^2
\leq\zeta(s e,t_*)-\zeta(0,t_*)
\leq-\frac{\lambda}{2}|s|^{\kappa_j}.
\end{equation*}
It follows that
$\mathrm{C} |s|^{2-\kappa_j}\geq\lambda/2$, which is impossible
as $s\to0$. Thus no such lower touching test exists.

Otherwise, $\mathrm{M}$ is positive definite, because
it is non-singular. Therefore,
$a>\mathrm{tr}(\mathrm{M})>0$.
At a zero-gradient test, the diffusion term vanishes
when $\gamma_j>0$, and the required inequality follows
from $\zeta_t(0,t_*)=a>0$.

\smallskip
\noindent
(ii) $\gamma_j=0$.
In this case, $\kappa_j=2$ and
$P_j(x)=\frac12\langle\mathrm{M}x,x\rangle$.
The lower touching inequality yields
$D^2\zeta(0,t_*)\leq\mathrm{M}$.
For every unit vector $e$, the matrix
$\textbf{I}\mathrm{d_n}+(p_j-2)e\otimes e$
is positive definite, since its eigenvalues are
$1$ and $p_j-1>0$. Consequently,
\begin{equation*}
\Delta\zeta(0,t_*)
+(p_j-2)\langle D^2\zeta(0,t_*)e,e\rangle
\leq\mathrm{tr}(\mathrm{M})
+|p_j-2|\|\mathrm{M}\|.
\end{equation*}
The same upper bound holds for the zero-gradient
envelope. Since $p_j\to2$, for all sufficiently large $j$,
\begin{equation*}
a-\mathrm{tr}(\mathrm{M})-|p_j-2|\|\mathrm{M}\|>0.
\end{equation*}
This proves the super-solution inequality.

\smallskip
\noindent
(iii) $-1<\gamma_j<0$.
In this case, $\kappa_j>2$.
By the admissibility of $\zeta$ at a zero-gradient point,
there exists $\rho\in C^2([0,\infty))$ satisfying
\begin{equation*}
\rho(0)=\rho'(0)=\rho''(0)=0,
\quad
\rho''(s)>0\quad\text{for }s>0,
\end{equation*}
such that, near $x=0$,
\begin{equation*}
|\zeta(x,t_*)-\zeta(0,t_*)|
\leq\rho(|x|),
\end{equation*}
and
\begin{equation}\label{eq:heat-admissible-radial-limit}
(\rho'(s))^{\gamma_j}
\left[
(p_j-1)\rho''(s)
+(n-1)\frac{\rho'(s)}{s}
\right]
\longrightarrow0
\quad\text{as }s\to0^+.
\end{equation}
Since $\rho'(s)>0$, $\rho''(s)>0$ and $p_j>1$,
both terms in \eqref{eq:heat-admissible-radial-limit}
are nonnegative. Hence
\begin{equation*}
(\rho'(s))^{\gamma_j}\rho''(s)\longrightarrow0.
\end{equation*}
For each fixed $j$, this gives
\begin{equation*}
\frac{d}{ds}(\rho'(s))^{1+\gamma_j}
=
(1+\gamma_j)(\rho'(s))^{\gamma_j}\rho''(s)
\longrightarrow0.
\end{equation*}
Given $\eta>0$, integration from zero therefore yields
\begin{equation*}
(\rho'(s))^{1+\gamma_j}\leq\eta s
\end{equation*}
for all sufficiently small $s>0$.
Integrating once more, we obtain
\begin{equation*}
\begin{aligned}
\rho(s)
&\leq
\eta^{1/(1+\gamma_j)}
\int_0^s\tau^{1/(1+\gamma_j)}\,d\tau\\
&=\frac{\eta^{1/(1+\gamma_j)}}{\kappa_j}s^{\kappa_j}.
\end{aligned}
\end{equation*}
Since $\eta>0$ is arbitrary, it follows that
\begin{equation}\label{eq:heat-admissible-flatness}
\rho(s)=o(s^{\kappa_j})
\quad\text{as }s\to0^+.
\end{equation}

If $\mathrm{M}$ has a negative eigenvalue $-\lambda$
with a corresponding unit eigenvector $e$, the lower
touching inequality gives
\begin{equation*}
-\rho(|s|)
\leq\zeta(se,t_*)-\zeta(0,t_*)
\leq-\frac{\lambda}{2}|s|^{\kappa_j}.
\end{equation*}
This contradicts \eqref{eq:heat-admissible-flatness}.
Thus no admissible lower touching test exists in this
case. Otherwise, $\mathrm{M}$ is positive definite and
$a>\mathrm{tr}(\mathrm{M})>0$, so the admissible
zero-gradient condition $\zeta_t(0,t_*)\geq0$ holds.

Combining cases (i), (ii), (iii) and
\eqref{eq:heat-strict-supersolution}, we conclude that
$\Psi_j$ is a continuous viscosity super-solution of the
corresponding homogeneous equation in $Q_r$ for all
sufficiently large $j$.

\medskip
\noindent {\bf Step 3. Conclusion.} Set
\begin{equation*}
v_j(x,t)
:=
w_j(x,t)+q_j\cdot x-\|f_j\|_{L^\infty(Q_1)}t.
\end{equation*}
Let $\zeta$ touch $v_j$ from above.
Then $\zeta-q_j\cdot x+\|f_j\|_{L^\infty(Q_1)}t$
touches $w_j$ from above at the same point.
The viscosity sub-solution inequality for $w_j$ gives
\begin{equation*}
\begin{aligned}
&\zeta_t+\|f_j\|_{L^\infty(Q_1)}
-|D\zeta|^{\gamma_j}
\left[
\Delta\zeta+(p_j-2)
\frac{
\langle D^2\zeta D\zeta,D\zeta\rangle
}{
|D\zeta|^2
}
\right]
\leq f_j
\end{aligned}
\end{equation*}
whenever $D\zeta\neq0$. Therefore,
\begin{equation*}\label{eq:heat-homogeneous-subsolution}
\begin{aligned}
&(v_j)_t
-|Dv_j|^{\gamma_j}
\left[
\Delta v_j+(p_j-2)
\frac{
\langle D^2v_jDv_j,Dv_j\rangle
}{
|Dv_j|^2
}
\right]\\
&\qquad\leq f_j-\|f_j\|_{L^\infty(Q_1)}
\leq0
\quad\text{in }Q_r
\end{aligned}
\end{equation*}
in the viscosity sense.
The same conclusion holds at zero gradient under the
corresponding viscosity convention; in particular,
the linear time correction preserves admissibility
in the singular case.

Moreover,
\begin{equation*}
\begin{aligned}
&\|v_j-(w_\infty+q_\infty\cdot x)\|_
  {L^\infty(\overline{Q_r})}\\
&\quad\leq
\|w_j-w_\infty\|_{L^\infty(\overline{Q_r})}
+r|q_j-q_\infty|
+r^2\|f_j\|_{L^\infty(Q_1)}
\longrightarrow0.
\end{aligned}
\end{equation*}
Together with \eqref{eq:heat-comparison-convergence},
this allows us to select $j$ large enough so that
\begin{equation*}
\begin{aligned}
\|v_j-(w_\infty+q_\infty\cdot x)\|_
  {L^\infty(\overline{Q_r})}
&\leq\frac{\delta r^2}{8},\\
\|\Psi_j-\Psi\|_{L^\infty(\overline{Q_r})}
&\leq\frac{\delta r^2}{8}.
\end{aligned}
\end{equation*}
It follows from \eqref{eq:heat-boundary-gap} that
\begin{equation*}
\begin{aligned}
v_j-\Psi_j
={}&
\bigl[v_j-(w_\infty+q_\infty\cdot x)\bigr]+\bigl[w_\infty+q_\infty\cdot x-\Psi\bigr]
+(\Psi-\Psi_j) \\
 \leq{}&
\frac{\delta r^2}{8}
-\frac{\delta r^2}{2}
+\frac{\delta r^2}{8}\\
={}&-\frac{\delta r^2}{4}
\quad\text{on }\partial_{\mathrm{par}}Q_r.
\end{aligned}
\end{equation*}
In particular,
\begin{equation*}\label{eq:heat-comparison-boundary-order}
v_j\leq\Psi_j-\frac{\delta r^2}{8}
\quad\text{on }\partial_{\mathrm{par}}Q_r.
\end{equation*}

By {\bf Step~2}, we have that $\Psi_j-\frac{\delta r^2}{8} $
is still a super-solution of the same homogeneous equation.
Using the comparison principle (\cite[Theorem 5.2]{IJS19}), we get
\begin{equation}\label{eq:heat-comparison-conclusion}
v_j\leq\Psi_j-\frac{\delta r^2}{8}
\quad\text{on }\overline{B_r}\times[-r^2,0].
\end{equation}

Evaluating \eqref{eq:heat-comparison-conclusion}
at $(0,0)$ and using $\Psi_j(0,0)=0$, we find
\begin{equation*}
w_j(0,0)=v_j(0,0)\leq-\frac{\delta r^2}{8}.
\end{equation*}
However, $w_j(0,0)\to w_\infty(0,0)=0$, which is
impossible, since $\delta>0$ and $r>0$ are fixed, independently of $j$.
This contradiction proves \eqref{Lemma3.2:eq130}
in {\bf Case~2} and completes the verification of \eqref{Section 3:eq3}.

Since $ \omega_{\infty} $ is caloric in $ Q_{3/4} $, taking $ h = \omega_{\infty} $ contradicts the assumption \eqref{Section 3:eq2}. This completes the proof of Lemma \ref{Section 3:lem2}.
\end{proof}

\vspace{2mm}

\section{One-step intrinsic improvement of flatness}\label{Section 4}

Fix the target exponent $\alpha\in(0,1)$ and define
\begin{equation}\label{eq:beta}
 \beta:=\frac{1+\alpha}{2}.
\end{equation}
Then $\alpha<\beta<1$.  Choose $\gamma_*>0$ sufficiently small so that
\begin{equation}\label{eq:gammastarconditions}
 \beta(1+\gamma_*)<1
\end{equation}
and
\begin{equation}\label{eq:timeconditions}
 \frac{\beta}{2+\beta\gamma_*}\ge\frac\alpha2,
 \quad
 \frac{1+\beta}{2+\beta\gamma_*}\ge\frac{1+\alpha}{2}.
\end{equation}
Such a choice is possible because $\beta>\alpha$.  For instance, it suffices to take
\[
 0<\gamma_*<\frac12\min\left\{
 \frac{1-\beta}{\beta},
 \frac{2(\beta-\alpha)}{\alpha\beta},
 \frac{2(\beta-\alpha)}{(1+\alpha)\beta}
 \right\}.
\]
Condition \eqref{eq:gammastarconditions} guarantees that the rescaled right-hand side decreases along the intrinsic iteration.  Conditions \eqref{eq:timeconditions} will be used only on the singular side $\gamma<0$ to convert intrinsic temporal exponents into the standard exponents appearing in Theorem~\ref{thm:main}.

We now turn the compactness lemma into a quantitative affine approximation.  Let $\eta\in(0,1)$ be fixed later and define $  \lambda:=\rho^{\beta} $. The intrinsic one-step cylinder is
\[
 Q_\rho^{\lambda}
 :=B_\rho\times(-\rho^{2-\beta\gamma},0].
\]

\begin{lemma}\label{lem:onestep}
Fix $M>0$, $\beta$ as in \eqref{eq:beta}, and $\gamma_*\le1/4$ satisfying \eqref{eq:gammastarconditions}. There exist $\rho=\rho(n,\beta,\gamma_*)\in(0,1/8)$ and $ \mathrm{C}_0 := \mathrm{C}_0(n)$ with the following property:

For each fixed $\eta\in(0,1)$ there exists $\varepsilon_0=\varepsilon_0(n,M,\beta,\gamma_*,\eta)>0$ such that, if $w$ solves \eqref{eq:shifted} in $Q_1$ and
\[
|q|\le M, \qquad  \norm w_{L^\infty(Q_1)}\le\eta,
\qquad  |p-2|+|\gamma|+\norm f_{L^\infty(Q_1)}\le\varepsilon_0,
\]
then an affine function $\ell(x)=a+b\cdot x$ satisfies
\begin{equation}\label{eq:onestepestimate}
\norm{w-\ell}_{L^\infty(Q_\rho^\lambda)}\le\eta\rho^{1+\beta}=\eta\rho\lambda,
\qquad \lambda:=\rho^\beta,
\end{equation}
and
\begin{equation}\label{eq:slopebound}
|a|+|b|\le \mathrm{C}_0\eta.
\end{equation}
In particular, $\rho$ is independent of both $M$ and $ \eta $.
\end{lemma}

\begin{proof}
Fix a constant $ \mathrm{C}_H > 0 $, and let $\kappa:=1-\beta(1+\gamma_*)>0$. Choose $\rho\in(0,1/8)$ so that $Q_\rho^{\rho^\beta}\subset Q_{1/4}$ for every $|\gamma|\le\gamma_*$ and
\begin{equation*}\label{Section4:eq6}
  \mathrm{C}_H(\rho^2+\rho^{2-\beta\gamma_*})\le\tfrac12\rho^{1+\beta}.
\end{equation*}
This is possible because $ 2-\beta\gamma_*=1+\beta+\kappa $; for example it suffices to impose $ 2 \mathrm{C}_H\rho^\kappa\le1/2 $ in addition to the cylinder inclusion.
Now fix $\eta$ and put
\begin{equation}\label{Section4:eq7}
\delta_\eta:=\tfrac18\eta\rho^{1+\beta},\qquad
\varepsilon_0:=\min\{\gamma_*,1/2,\varepsilon(n,M,\delta_\eta)\},
\end{equation}
where the last threshold is furnished by Lemma~\ref{Section 3:lem2}. That lemma yields a caloric function $ h $ with
\begin{equation}\label{Section4:eq8}
  \norm{w-h}_{L^\infty(Q_{1/2})}\le\delta_\eta,
\qquad \norm h_{L^\infty(Q_{1/2})}\le2\eta.
\end{equation}
By the standard interior $ C^{2,1}$ estimates for $ h $, we have
\begin{equation}\label{Section4:eq9}
|h(0,0)|+|Dh(0,0)|+\norm{D^2h}_{L^\infty(Q_{1/4})}
+\norm{h_t}_{L^\infty(Q_{1/4})}\le \mathrm{C}_H \eta.
\end{equation}
Setting $\ell(x):=h(0,0)+Dh(0,0)\cdot x$, then applying Taylor's formula, we obtain
\begin{equation*}\label{Section4:eq10}
\norm{h-\ell}_{L^\infty(Q_\rho^\lambda)}
\le \mathrm{C}_H\eta(\rho^2+\rho^{2-\beta\gamma})
\le\tfrac12\eta\rho^{1+\beta},
\end{equation*}
which, together with \eqref{Section4:eq7} and \eqref{Section4:eq8}, yields \eqref{eq:onestepestimate}. Here $ a:= h(0,0) $ and $ b:= Dh(0,0) $, then \eqref{eq:slopebound} follows automatically from \eqref{Section4:eq9}.
\end{proof}

\vspace{2mm}

\section{The critical iteration}\label{Section 5}

Fix $K_*:=2$. Use the scale $\rho$ from Lemma~\ref{lem:onestep}; apply that lemma with $ M=K_*$. Define
\[
 r_k:=\rho^k,
 \qquad
 \lambda_k:=\lambda^k=\rho^{k\beta}.
\]
Observe that
\begin{equation*}\label{eq:intrinsictime}
 r_k^2\lambda_k^{-\gamma}
 =r_k^{2-\beta\gamma}.
\end{equation*}

We work first with a normalized solution satisfying
\begin{equation}\label{eq:normalized}
 u(0,0)=0,
 \qquad \osc_{Q_1}u\leq\eta,
 \qquad |p-2|+|\gamma|+\norm{f}_{L^{\infty}(Q_1)}\leq\varepsilon_0.
\end{equation}
Since $u(0,0)=0$, the oscillation bound implies
$\norm{u}_{L^\infty(Q_1)}\leq\eta$.

\begin{proposition}\label{Section5:prop1}
There exist affine functions
\[
 L_k(x)=a_k+l_k\cdot x
\]
as long as
\begin{equation}\label{eq:criticalcondition}
 |l_j|\leq K_*\lambda_j
 \qquad (j=0,\ldots,k),
\end{equation}
such that
\begin{equation}\label{eq:flatnessk}
 \norm{u-L_k}_{L^{\infty}(Q_{r_k}^{\lambda_k})}
 \leq\eta r_k\lambda_k
 =\eta r_k^{1+\beta},
\end{equation}
and
\begin{equation}\label{eq:increments}
 |a_{k+1}-a_k|\leq \mathrm{C}_0\eta r_k\lambda_k,
 \qquad
 |l_{k+1}-l_k|\leq \mathrm{C}_0\eta\lambda_k.
\end{equation}
\end{proposition}

\begin{proof}
Set $L_0\equiv0$.  Suppose \eqref{eq:criticalcondition} and \eqref{eq:flatnessk} hold at level $k$. Now we introduce an auxiliary function as follows:
\[
 w_k(x,t):=
 \frac{
 u(r_kx,r_k^2\lambda_k^{-\gamma}t)
 -a_k-r_kl_k\cdot x
 }{r_k\lambda_k}.
\]
Then
\[
 \norm{w_k}_{L^{\infty}(Q_1)}\leq\eta.
\]
A direct computation gives
\[
 (w_k)_t-|Dw_k+q_k|^{\gamma}
 \mathcal A_p(Dw_k+q_k):D^2w_k=f_k,
\]
where
\[
 \mathcal{A}_{p}(\xi):= \textbf{I}\mathrm{d_{n}} + (p-2) \frac{\xi \otimes \xi}{|\xi|^{2}}    \quad   \text{and}  \quad         q_k:=\frac{l_k}{\lambda_k}
\]
and
\begin{equation}\label{eq:fk}
 f_k(x,t)
 :=r_k\lambda_k^{-(1+\gamma)}
 f(r_kx,r_k^2\lambda_k^{-\gamma}t).
\end{equation}
Since $|q_k|\leq K_*$ and
\[
 r_k\lambda_k^{-(1+\gamma)}
 =\rho^{k[1-\beta(1+\gamma)]}
 \leq1
\]
by \eqref{eq:gammastarconditions}, we have
\[
 \norm{f_k}_{\infty}\leq\varepsilon_0.
\]
Lemma~\ref{lem:onestep} therefore gives an affine function
\[
 \ell_k(x)=A_k+B_k\cdot x,
 \qquad |A_k|+|B_k|\leq \mathrm{C}_0\eta,
\]
such that
\[
 \norm{w_k-\ell_k}_{L^{\infty}(Q_\rho^{\lambda})}
 \leq\eta\rho\lambda.
\]
Scaling back, define
\[
 a_{k+1}:=a_k+r_k\lambda_kA_k,
 \qquad
 l_{k+1}:=l_k+\lambda_kB_k.
\]
Noticing that
\[
 r_k^2\lambda_k^{-\gamma}\rho^2\lambda^{-\gamma}
 =r_{k+1}^2\lambda_{k+1}^{-\gamma},
\]
then we obtain \eqref{eq:flatnessk} at level $k+1$, and \eqref{eq:increments} follows immediately.
\end{proof}

The iteration either continues indefinitely, with
$|l_k|\le K_*\lambda_k$ at every level, or reaches a first
level $k_0$ at which $|l_{k_0}|>K_*\lambda_{k_0}$.
We first treat the infinite iteration and then describe the reduction associated with the first stopping scale.

\subsection{Infinite critical iteration}\label{Section 5.1}

Suppose \eqref{eq:criticalcondition} holds for every $k$. Then
\[
 |l_k|\leq K_*\lambda_k\to0,
\]
and by \eqref{eq:increments} both $(a_k)_{k} $ and $(l_k)_{k} $ are Cauchy sequences.  Hence
\[
 l_k\to0,
 \qquad
 a_k\to u(0,0)=0.
\]
Moreover, summing the increments and using \eqref{eq:flatnessk},
\begin{equation*}\label{eq:criticalpointwise}
 |u(x,t)-u(0,0)|
 \leq \mathrm{C} r^{1+\beta}
\end{equation*}
whenever
\[
 |x|\leq r,
 \qquad
 |t|\leq r^{2-\beta\gamma}.
\]
Equivalently,
\begin{equation*}\label{eq:intrinsicexpansion}
 |u(x,t)-u(0,0)|
 \leq \mathrm{C} \left(
 |x|^{1+\beta}
 +|t|^{\frac{1+\beta}{2-\beta\gamma}}
 \right).
\end{equation*}
In particular $Du(0,0)=0$.

\subsection{The first stopping scale}
\label{subsec:first-stop}

Suppose that the critical slope condition fails after finitely
many steps, and let $k_0\ge1$ be the first index at which it fails.
Thus,
\begin{equation*}\label{eq:first-stopping-index}
 |l_j|\le K_*\lambda_j
 \quad   \text{for }  0  \le j<k_0,
 \qquad
 |l_{k_0}|>K_*\lambda_{k_0}.
\end{equation*}

Since the critical slope condition holds at level $k_0-1$,
the iteration constructs $L_{k_0}$ and yields
\begin{equation*}\label{eq:flatness-at-stop}
 \|u-L_{k_0}\|_{L^\infty(Q_{r_{k_0}}^{\lambda_{k_0}})}
 \le \eta r_{k_0}\lambda_{k_0}.
\end{equation*}
In particular, the failure of the slope condition at level
$k_0$ does not affect the approximation already obtained
at that level.

The minimality of $k_0$ also gives an upper bound for the
normalized slope. Indeed, by \eqref{eq:increments},
\[
 |l_{k_0}|
 \le |l_{k_0-1}|+ \mathrm{C}_0\eta\lambda_{k_0-1}
 \le (K_*+ \mathrm{C}_0\eta)\lambda_{k_0-1}.
\]
Consequently, setting
\[
 q_{k_0}:=\frac{l_{k_0}}{\lambda_{k_0}},
 \qquad
 M_*:=\rho^{-\beta}(K_*+ \mathrm{C}_0),
\]
and requiring $\eta\le1$, we obtain
\begin{equation}\label{eq:slope-at-stop}
 2=K_*<|q_{k_0}|
 \le \rho^{-\beta}(K_*+ \mathrm{C}_0\eta)
 \le M_*.
\end{equation}
The constant $M_*$ is independent of $k_0$ and of the
subsequent choice of $\eta$.

Define the rescaled function $ \omega_{k_{0}}$ as follows
\[
 \omega_{k_0}(x,t):=
 \frac{
 u(r_{k_0}x,r_{k_0}^2\lambda_{k_0}^{-\gamma}t)
 -a_{k_0}-r_{k_0}l_{k_0}\cdot x
 }{r_{k_0}\lambda_{k_0}}.
\]
It satisfies
\[
 (\omega_{k_0})_t
 -|D\omega_{k_0}+q_{k_0}|^\gamma
 \mathcal A_p(D\omega_{k_0}+q_{k_0}):D^2\omega_{k_0}
 =f_{k_0}
 \quad   \text{in }  \quad   Q_1,
\]
where $f_{k_0}$ is given by \eqref{eq:fk}. Moreover,
\begin{equation}\label{eq:bounds-at-stop}
 \|\omega_{k_0}\|_{L^\infty(Q_1)}\le\eta,
 \qquad
 \|f_{k_0}\|_{L^\infty(Q_1)}\le\varepsilon_0.
\end{equation}

We have therefore reduced the finite-stopping case to a
small deviation from a plane whose normalized slope lies
in the fixed range \eqref{eq:slope-at-stop}.

At the first stopping scale, the rescaled solution is a small
deviation from an affine function whose slope lies in a fixed
range. We first establish a spatial Lipschitz estimate for this
deviation. This places the test gradients of the corresponding
solution in a fixed annulus, where a perturbative argument yields
the prescribed regularity exponent $\beta$. All estimates obtained below are independent of the stopping index.

\begin{lemma}\label{lem:noncriticalLip}
Fix $M>2$. There exist two numbers
\[
 \varepsilon_*=\varepsilon_*(n,M)>0,
 \qquad
 \eta_*=\eta_*(n,M)>0
\]
with the following property: assume
\[
 |p-2|+|\gamma|\leq \varepsilon_*,
 \qquad
 2\leq |q|\leq M,
\]
and let $ \omega \in C^{0}(Q_1)$ be a viscosity solution of
\begin{equation}\label{eq:shifted-noncritical}
 \omega_t-|Dw+q|^\gamma
 \left[
 \Delta \omega+
 (p-2)\left\langle
 D^2\omega \frac{D\omega+q}{|D\omega+q|},
 \frac{D\omega+q}{|D\omega+q|}
 \right\rangle
 \right]= f
 \quad \hbox{in } Q_1.
\end{equation}
If
\[
 \|\omega\|_{L^\infty(Q_1)}
 +\|f\|_{L^\infty(Q_1)}
 \leq \eta_*,
\]
then, for every fixed $t\in[-(3/4)^2,0]$,
\begin{equation}\label{eq:unit-lip}
 |\omega(x,t)-\omega(y,t)|\leq |x-y|,
 \qquad x,y\in B_{3/4}.
\end{equation}
Consequently, if
\[
 v(x,t):=\omega(x,t)+q\cdot x,
\]
then every $C^{2,1}$ test function $\phi$ touching $v$ from above or from
below at a point of $Q_{3/4}$ satisfies
\begin{equation}\label{eq:annulus-test-gradients}
 1\leq |D\phi|\leq M+1.
\end{equation}
The constants are uniform for $ \gamma > -1 $ sufficiently close to zero.
\end{lemma}

\begin{proof}
We divide the proof into four steps.

\medskip
\noindent
\textit{Step 1}.
Set
\[
 h(x,t):=\omega(x,t)+q\cdot x.
\]
Then $h$ solves
\[
 h_t-|Dh|^\gamma\Delta_p^N h= f
 \quad   \hbox{in }  \quad  Q_1.
\]
We first claim that, after decreasing $\varepsilon_*$ if necessary, there
is a number $ \mathrm{C}_0 = \mathrm{C}_0(n,M)$ such that
\begin{equation}\label{eq:coarse-lip}
 |\omega(x,t)-\omega(y,t)|
 \leq \mathrm{C}_0|x-y|,
 \qquad
 x,y\in B_{7/8},\quad t\in[-(7/8)^2,0].
\end{equation}
Indeed, this local spatial Lipschitz estimate \eqref{eq:coarse-lip} proved in \cite[Lemma~6.2]{A20}. What matters here is that its constant
can be chosen uniformly when $(p,\gamma)$ ranges in a sufficiently small
compact neighbourhood of $(2,0)$.  This can be seen directly from the
proof: fix, for instance, the auxiliary exponents used in the
Ishii--Lions argument to be
\[
 \bar\beta=\frac12,
 \qquad
 \nu_0=1+\frac{\bar\beta}{2}=\frac54.
\]
If
\[
 \frac32\leq p\leq\frac52,
 \qquad
 |\gamma|\leq\frac14,
\]
then
\[
 \min\{1,p-1\}\geq\frac12,\qquad
 \max\{1,p-1\}\leq\frac32,
 \qquad
 \frac1{1+\gamma}\in\left[\frac45,\frac43\right].
\]
Thus every structural constant occurring in that doubling-variable proof
is bounded on this parameter box.  Since $ \|h\|_{L^\infty(Q_1)}
 \leq M+1 $ when $\eta_*\leq1$, the resulting bound depends only on $ n $ and $ M $. Subtracting the affine function $q\cdot x$ gives \eqref{eq:coarse-lip}.

Now combining \eqref{eq:coarse-lip} with
$\operatorname{osc}_{Q_1} \omega \leq2\eta_*$ gives a small H\"older modulus.
Indeed, for $ r=|x-y| $,
\[
 |\omega(x,t)-\omega(y,t)|
 \leq
 \min\{2\eta_*, \mathrm{C}_0r\}
 \leq
 H r^{\bar\beta},
\]
where
\begin{equation*}\label{eq:H-small}
 H:=(2\eta_*)^{1-\bar\beta} \mathrm{C}_0^{\bar\beta}.
\end{equation*}
In particular,
\[
 H    \to  0
 \quad   \hbox{as }  \eta_* \rightarrow 0.
\]

\medskip
\noindent
\textit{Step 2}.
For $\xi\neq0$, define
\begin{equation*}\label{eq:B-matrix}
 B_{p,\gamma}(\xi)
 :=
 |\xi|^\gamma
 \left(
 \textbf{I}\mathrm{d_{n}}+(p-2)\widehat\xi\otimes\widehat\xi
 \right),
 \qquad
 \widehat\xi:=\frac{\xi}{|\xi|}.
\end{equation*}
Fix the compact annulus
\[
 \mathcal A_M
 :=
 \left\{
 \xi\in\mathbb R^n:
 \frac34\leq|\xi|\leq M+\frac32
 \right\}.
\]

We now make the following claim:

{\it Claim.} There is a constant $ \mathrm{C}_M: = \mathrm{C}(n,M)$ such that, for
$|p-2|+|\gamma|\leq\varepsilon_*$,
\begin{equation}\label{eq:B-estimates}
 \frac12 \textbf{I}\mathrm{d_{n}} \leq B_{p,\gamma}(\xi)\leq 2 \textbf{I}\mathrm{d_{n}},
 \qquad
 \|D_\xi B_{p,\gamma}(\xi)\|
 \leq \mathrm{C}_M \bigl(|p-2|+|\gamma|\bigr)
 \quad\hbox{for }\xi\in\mathcal A_M.
\end{equation}

In fact, we set $r:=|\xi|$. Since
\[
\frac34\le r\le M+\frac32,
\]
the variable $r$ stays in a compact subset of $(0,\infty)$.
For fixed $r$, apply the mean value theorem to the function $ r^\gamma-1 $, then we have, for some $\theta\in(0,1)$,
\[
r^\gamma-1
=
\gamma r^{\theta\gamma}\log r.
\]
Hence, if $|\gamma|\le \frac14$,
\begin{equation}\label{eq:power-close}
\bigl|r^\gamma-1\bigr|
\le
|\gamma|
\sup_{\substack{
3/4\le \rho\le M+3/2\\
|s|\le 1/4
}}
\rho^s|\log \rho|
\le \mathrm{C}_M|\gamma|.
\end{equation}
In particular,
\begin{equation}\label{eq:power-bounded}
r^\gamma\le \mathrm{C}_M
\qquad
\text{for all }\xi\in\mathcal A_M,
\quad |\gamma|\le \frac14.
\end{equation}

Now write
\[
B_{p,\gamma}(\xi)- \textbf{I}\mathrm{d_{n}}
=
\bigl(|\xi|^\gamma-1\bigr) \textbf{I}\mathrm{d_{n}}
+
(p-2)|\xi|^\gamma
\widehat\xi\otimes\widehat\xi.
\]
Since
\[
\|\widehat\xi\otimes\widehat\xi\|=1,
\]
we obtain from \eqref{eq:power-close}--\eqref{eq:power-bounded}
\begin{align*}
\|B_{p,\gamma}(\xi)- \textbf{I}\mathrm{d_{n}}\|
&\le
\bigl||\xi|^\gamma-1\bigr|
+
|p-2|\,|\xi|^\gamma
\nonumber\\
&\le
\mathrm{C}_M\bigl(|\gamma|+|p-2|\bigr).
\label{eq:B-close-I}
\end{align*}

Choose $\varepsilon_*=\varepsilon_*(n,M)>0$ sufficiently small so that
\[
\mathrm{C}_M\varepsilon_*\le \frac12.
\]
Then, whenever
\[
|p-2|+|\gamma|\le \varepsilon_*,
\]
we have
\[
\|B_{p,\gamma}(\xi)- \textbf{I}\mathrm{d_{n}}\|\le \frac12.
\]
Consequently, for every $z\in\mathbb R^n$,
\[
\frac12|z|^2
\le
\langle B_{p,\gamma}(\xi)z,z\rangle
\le
\frac32|z|^2
\le
2|z|^2,
\]
which implies
\begin{equation}\label{eq:ellipticity}
\frac{1}{2} \textbf{I}\mathrm{d_{n}}  \le B_{p,\gamma}(\xi)\le 2 \textbf{I}\mathrm{d_{n}}.
\end{equation}
This complete the proof of the first estimate of \eqref{eq:B-estimates}.

\medskip

\noindent
To verify the second estimate, we let
\[
P(\xi):=\widehat\xi\otimes\widehat\xi.
\]
Then
\[
B_{p,\gamma}(\xi)
=
|\xi|^\gamma
\bigl(\textbf{I}\mathrm{d_{n}}+(p-2)P(\xi)\bigr).
\]
Differentiating with respect to $\xi$ gives
\begin{equation}\label{eq:DB-split}
D_\xi B_{p,\gamma}(\xi)
=
D_\xi(|\xi|^\gamma)
\otimes
\bigl(\textbf{I}\mathrm{d_{n}}+(p-2)P(\xi)\bigr)
+
(p-2)|\xi|^\gamma D_\xi P(\xi).
\end{equation}
The radial factor satisfies
\[
D_\xi(|\xi|^\gamma)
=
\gamma|\xi|^{\gamma-2}\xi,
\]
and therefore
\begin{equation}\label{eq:Dpower}
\bigl|D_\xi(|\xi|^\gamma)\bigr|
=
|\gamma|\,|\xi|^{\gamma-1}
\le
\mathrm{C}_M|\gamma|,
\end{equation}
where we have again used the fact that
\begin{equation*}
  |\xi|\in \left[\frac34,M+\frac32\right]  \quad  \text{and} \quad   |\gamma|\le 1/4.
\end{equation*}

It remains to estimate the derivative of $P(\xi)$. In components,
\[
\widehat\xi_i=\frac{\xi_i}{|\xi|},  \quad  i=1,2,\cdots, n,
\]
so
\begin{equation*}\label{eq:Dhatxi}
\partial_{\xi_k}\widehat\xi_i
=
\frac{\delta_{ik}}{|\xi|}
-
\frac{\xi_i\xi_k}{|\xi|^3}
=
\frac{1}{|\xi|}
\bigl(
\delta_{ik}-\widehat\xi_i\widehat\xi_k
\bigr), \quad  k, i=1,2,\cdots, n.
\end{equation*}
Equivalently,
\[
D_\xi\widehat\xi
=
\frac1{|\xi|}
\left(
\textbf{I}\mathrm{d_{n}} - \widehat\xi\otimes\widehat\xi
\right),
\]
and hence
\[
\|D_\xi\widehat\xi\|
\le
\frac{\mathrm{C}}{|\xi|}.
\]

Using the product rule,
\[
D_\xi P(\xi)
=
D_\xi\widehat\xi\otimes\widehat\xi
+
\widehat\xi\otimes D_\xi\widehat\xi,
\]
so that
\begin{equation}\label{eq:DP}
\|D_\xi P(\xi)\|
\le
\frac{\mathrm{C}}{|\xi|}
\le \mathrm{C}
\quad
\text{on }  \mathcal  A_M.
\end{equation}

Combining \eqref{eq:DB-split}, \eqref{eq:Dpower},
\eqref{eq:DP}, and the uniform bound \eqref{eq:power-bounded}, we infer
\begin{align*}
\|D_\xi B_{p,\gamma}(\xi)\|
&\le
\mathrm{C}_M|\gamma|\bigl(1+|p-2|\bigr)
+
\mathrm{C}_M|p-2|
\\
&\le
\mathrm{C}_M\bigl(|\gamma|+|p-2|\bigr),
\end{align*}
after decreasing $\varepsilon_*$ once more, if necessary.
Together with \eqref{eq:ellipticity}, this proves \eqref{eq:B-estimates}.

\medskip
\noindent
\textit{Step 3}.
Fix
\[
 \nu:=1+\frac{\bar\beta}{4}=\frac98.
\]
Choose $\kappa_0>0$, depending only on $\nu$, so small that
\[
 \nu\kappa_0 2^{\nu-1}\leq\frac14,
\]
and put
\[
 \varphi(s):=s-\kappa_0s^\nu,
 \qquad 0\leq s\leq2.
\]
Then
\begin{equation}\label{eq:phi-der}
 \frac34\leq\varphi'(s)\leq   1
 \quad   \text{and}   \quad
 \varphi''(s)
 =
 -\mathrm{c}_\nu s^{\nu-2},
 \quad
 \mathrm{c}_\nu:=\kappa_0\nu(\nu-1)>0.
\end{equation}

Let $x_0,y_0\in B_{3/4}$ and
$t_0\in[-(3/4)^2,0)$.  We prove
\[
 \omega(x_0,t_0)- \omega(y_0,t_0)
 \leq |x_0-y_0|.
\]
Choose
\[
 L_1:=4^5\eta_*
\]
and consider
\begin{align*}
 \Psi(x,y,t)
 :=\;&
 \omega(x,t)-\omega(y,t)-\varphi(|x-y|)
 \notag\\
 &-\frac{L_1}{2}
 \left(
 |x-x_0|^2+|y-y_0|^2+|t-t_0|^2
 \right).
 \label{eq:Psi}
\end{align*}
Suppose, toward a contradiction, that this maximum is positive and is
attained at $(\bar x,\bar y,\bar t) \in \overline B_{7/8}\times\overline B_{7/8}
 \times[-(7/8)^2,0] $.  Since $\|w\|_\infty\leq\eta_*$ and
$L_1=4^5\eta_*$, positivity yields
\[
 |\bar x-x_0|+|\bar y-y_0|+|\bar t-t_0|
 \leq\frac18
\]
after an inessential enlargement of the numerical constant $4^5$.
Thus the maximum lies in an interior cylinder.  Moreover $\bar x\neq\bar y$.
Set
\[
 s:=|\bar x-\bar y|,
 \qquad
 e:=\frac{\bar x-\bar y}{s}.
\]

The H\"older estimate from {\it Step~1} and the positivity of $\Psi$ imply
\[
 \frac{L_1}{2}
 \bigl(
 |\bar x-x_0|^2+|\bar y-y_0|^2
 \bigr)
 \leq Hs^{\bar\beta}.
\]
After taking $\eta_*$ smaller, we may suppose $L_1\leq H$, and hence
\begin{equation}\label{eq:linear-penalty}
 L_1|\bar x-x_0|
 +L_1|\bar y-y_0|
 \leq C H s^{\bar\beta/2}.
\end{equation}

We now apply the parabolic version of Ishii-Lions Lemma; see \cite[Theorem 8.3]{CIL92}.  More precisely, write
\[
 \begin{aligned}
 U(x,t)&:=w(x,t)-\frac{L_1}{2}|x-x_0|^2
              -\frac{L_1}{2}|t-t_0|^2,\\
 V(y,t)&:=w(y,t)+\frac{L_1}{2}|y-y_0|^2.
 \end{aligned}
\]
Then $(\bar x,\bar y,\bar t)$ is a positive maximum point of
\[
 U(x,t)-V(y,t)-\varphi(|x-y|).
\]
Set
\[
 Z:=\varphi''(s)e\otimes e
 +\frac{\varphi'(s)}{s}(\textbf{I}\mathrm{d_{n}}-e\otimes e),
\]
so that
\[
 Ze=\varphi''(s)e,
 \qquad
 Z^2e=(\varphi''(s))^2e.
\]
Choose the Ishii parameter through
\begin{equation}\label{eq:tau-choice}
 \tau
 :=4\left(
 |\varphi''(s)|+\frac{\varphi'(s)}{s}
 \right).
\end{equation}
By the Ishii-Lions Lemma \cite[Theorem 8.3]{CIL92}, we have that there exists symmetric matrices $ \mathrm{X},\mathrm{Y} $ and a number $ b $ such that
\[
 \left(
 b+L_1(\bar t-t_0),
 a_1,
 \mathrm{X}+L_1 \textbf{I}\mathrm{d_{n}}
 \right)
 \in\overline{\mathcal P}^{2,+}w(\bar x,\bar t),
\]
\[
 \left(
 b,
 a_2,
 \mathrm{Y}-L_1 \textbf{I}\mathrm{d_{n}}
 \right)
 \in\overline{\mathcal P}^{2,-}w(\bar y,\bar t),
\]
where
\[
 a_1
 =\varphi'(s)e+L_1(\bar x-x_0),
 \qquad
 a_2
 =\varphi'(s)e-L_1(\bar y-y_0),
\]
and the matrix inequalities include
\begin{equation}\label{eq:ishii-lower}
 -\bigl(\tau+2\|\mathrm{Z}\|\bigr)
 \begin{pmatrix}\textbf{I}\mathrm{d_{n}} & 0 \\ 0 & \textbf{I}\mathrm{d_{n}}   \end{pmatrix}
 \leq
 \begin{pmatrix}\mathrm{X}&0\\0&-\mathrm{Y}\end{pmatrix} \leq \begin{pmatrix}
 \mathrm{Z} +\frac{2}{\tau}\mathrm{Z}^2&-\mathrm{Z}-\frac{2}{\tau}\mathrm{Z}^2\\
 -\mathrm{Z}-\frac{2}{\tau}\mathrm{Z}^2 & \mathrm{Z} +\frac{2}{\tau}\mathrm{Z}^2
 \end{pmatrix}.
\end{equation}

Then by the matrix inequalities \eqref{eq:ishii-lower} and \eqref{eq:tau-choice}, we get
\begin{equation}\label{eq:XY-norm-final}
 \|\mathrm{X}\|+\|\mathrm{Y}\|
 \leq
 \mathrm{C}\left(
 \frac{\varphi'(s)}s+|\varphi''(s)|
 \right)
 \leq
 \mathrm{C}\bigl(s^{-1}+s^{\nu-2}\bigr).
\end{equation}

Taking the quadratic form of \eqref{eq:ishii-lower} on $(\vec{\zeta},\vec{\zeta})$
shows, for every $ \vec{\zeta} \in \mathbb R^n$, that
\[
 \langle(\mathrm{X}-\mathrm{Y}) \vec{\zeta},\vec{\zeta} \rangle\leq 0;
\]
hence
\begin{equation}\label{eq:XY-est}
 \mathrm{X}- \mathrm{Y} \leq 0.
\end{equation}
On the other hand, testing \eqref{eq:ishii-lower} on $ (\vec{e},-\vec{e}) $ gives
\[
 \langle(\mathrm{X}-\mathrm{Y}) \vec{e}, \vec{e} \rangle
 \leq
 4\left[
 \varphi''(s)+\frac{2}{\tau}(\varphi''(s))^2
 \right].
\]
Since \eqref{eq:tau-choice} implies
$\tau\geq4|\varphi''(s)|$ and $\varphi''(s)<0$, we have
\[
 \varphi''(s)+\frac{2}{\tau}(\varphi''(s))^2
 \leq\frac12\varphi''(s).
\]
Consequently,
\begin{equation}\label{eq:negative-direction}
 \langle(\mathrm{X}-\mathrm{Y})e,e\rangle
 \leq2\varphi''(s)
 =-2 \mathrm{c}_\nu s^{\nu-2}.
\end{equation}
This is the strict negative direction needed below; in particular, the
positive quadratic term in the theorem of sums has been retained rather than discarded.

By \eqref{eq:phi-der} and \eqref{eq:linear-penalty}, taking $\eta_*$ small enough gives
\[
 |a_i|\leq\frac98,
 \qquad i=1,2.
\]
Consequently, with
\[
 \xi_i:=q+a_i,
\]
we have $\xi_i\in\mathcal A_M$.  Moreover,
\begin{equation}\label{eq:xi-difference}
 |\xi_1-\xi_2|
 =
 |a_1-a_2|
 \leq \mathrm{C} Hs^{\bar\beta/2}.
\end{equation}

Write
\[
 B_i:=B_{p,\gamma}(\xi_i).
\]
The viscosity inequalities for \eqref{eq:shifted-noncritical}, evaluated
with the two semijets above, contain the localization matrices
$\mathrm{X} + L_1 \textbf{I}\mathrm{d_{n}} $ and $\mathrm{Y} -L_1 \textbf{I}\mathrm{d_{n}} $ and the time error $L_1(\bar t-t_0)$.
Using \eqref{eq:B-estimates} to bound these localization contributions,
and then subtracting the inequality at $(\bar y,\bar t)$ from the one at
$(\bar x,\bar t)$, yields
\begin{equation}\label{eq:subtract-visc}
 0
 \leq
 \operatorname{tr}\bigl(B_1(\mathrm{X}-\mathrm{Y})\bigr)
 +
 \operatorname{tr}\bigl((B_1-B_2)\mathrm{Y}\bigr)
 +
 \mathrm{C}_nL_1
 +
 2\|f\|_\infty.
\end{equation}
By \eqref{eq:XY-est}, \eqref{eq:negative-direction} and \eqref{eq:B-estimates}, then it reads
\begin{equation}
 \operatorname{tr}(B_1(\mathrm{X}-\mathrm{Y}))
 \leq
 \frac12\operatorname{tr}(\mathrm{X}-\mathrm{Y})
 \leq
 \frac12\langle (\mathrm{X}-\mathrm{Y})e,e\rangle
 \leq
 -\mathrm{c}_\nu s^{\nu-2}.
\end{equation}
On the other hand, by \eqref{eq:B-estimates},
\eqref{eq:xi-difference}, and \eqref{eq:XY-norm-final},
\begin{align}
\begin{split}
 \left|
 \operatorname{tr}\bigl((B_1-B_2)\mathrm{Y}\bigr)
 \right|
 &\leq
 \mathrm{C}_M\varepsilon_*\,|\xi_1-\xi_2|\,\|\mathrm{Y}\|\\
 &\leq
 \mathrm{C}_M\varepsilon_* H
 s^{\bar\beta/2}
 \bigl(s^{-1}+s^{\nu-2}\bigr).
 \end{split}
\end{align}
Because
$
 1+\frac{\bar\beta}{2}-\nu
 =
 \frac{\bar\beta}{4}>0 $
and $0<s\leq2$, the last quantity is bounded by
\begin{equation}\label{eq:bad-matrix}
 \mathrm{C}_M\varepsilon_*H\,s^{\nu-2}.
\end{equation}
Combining \eqref{eq:subtract-visc}--\eqref{eq:bad-matrix} gives
\begin{equation}\label{eq:conclu1}
 0
 \leq
 \left(
 -\mathrm{c}_\nu+ \mathrm{C}_M\varepsilon_*H
 \right)s^{\nu-2}
 +
 \mathrm{C}_nL_1+2\eta_*.
\end{equation}
Choose first $\varepsilon_*$ so that \eqref{eq:B-estimates} holds, and then
choose $\eta_*$ sufficiently small that
\[
 \mathrm{C}_M\varepsilon_*H\leq\frac{c_\nu}{4}
\]
and
\[
 \mathrm{C}_nL_1+2\eta_*
 \leq
 \frac{\mathrm{c}_\nu}{4}2^{\nu-2}.
\]
Since $\nu<2$ and $s\leq2$, $ s^{\nu-2}\geq2^{\nu-2} $, hence \eqref{eq:conclu1} reduces
\begin{equation*}
  0 \leq -\frac{\mathrm{c}_{\nu}}{2} s^{\nu-2} < 0
\end{equation*}
which is a contradiction. Therefore
\[
 w(x_0,t_0)-w(y_0,t_0)
 \leq |x_0-y_0|.
\]
Interchanging $x_0$ and $y_0$, and then letting $t_0 \rightarrow 0$, proves
\eqref{eq:unit-lip}.

\medskip
\noindent
\textit{Step 4}.
If $\phi$ touches $v$ from above or below at $(x,t)\in Q_{3/4}$, then $ \phi(\,\cdot\,,\,\cdot\,)-q\cdot x $ touches the spatially $ 1 $-Lipschitz function $ w $ at the same point.
Hence
\[
 |D\phi-q|\leq1.
\]
Since $2\leq|q|\leq M$, this yields
\[
 1\leq|D\phi|\leq M+1,
\]
which is \eqref{eq:annulus-test-gradients}.
\end{proof}

With the test gradients now confined to a fixed annulus away from zero, we turn to the improved regularity in the noncritical regime.

\begin{lemma}
\label{lem:smoothimproved}
Let $0<\beta<1$ and $0<m<M<\infty$.  There exist
\[
 \varepsilon_s=\varepsilon_s(n,\beta,m,M)>0,
 \qquad
 \mathrm{C}_s = \mathrm{C}_s(n,\beta,m,M)>0
\]
such that the following properties hold: let $v\in C(Q_1)$ be a viscosity solution
of
\begin{equation*}\label{eq:smooth-original}
 v_t-|Dv|^\gamma\Delta_p^{\mathrm{N}} v = f
 \quad   \hbox{in }   Q_1,
\end{equation*}
and assume
\[
 |p-2|+|\gamma|+\|f\|_{L^\infty(Q_1)}
 \leq\varepsilon_s.
\]
In addition, assume that every $C^{2,1}$ test function $\phi$ touching
$v$ from above or below in $Q_1$ satisfies
\begin{equation}\label{eq:test-annulus-smooth}
 m\leq|D\phi|\leq M.
\end{equation}
Then
\[
 v\in C^{1+\beta,\frac{1+\beta}{2}}(Q_{1/2})
\]
with the estimate
\begin{equation}\label{eq:smooth-est}
 \|v\|_{C^{1+\beta,\frac{1+\beta}{2}}(Q_{1/2})}
 \leq
 \mathrm{C}_s\bigl(1+\|v\|_{L^\infty(Q_1)}\bigr).
\end{equation}
\end{lemma}

\begin{proof}
The proof is a complete perturbative improvement-of-flatness argument. More precisely, we proceed in the following three steps:

\medskip
\noindent
\textit{Step 1}.
Choose
\[
 \chi\in C_c^\infty((0,\infty)),
 \qquad
 0\leq\chi\leq1,
\]
so that
\begin{equation*}
\begin{cases}
    \chi(s) = 1, & \text{for } m \leq s \leq M, \\[2pt]
    \chi(s) = 0, & \text{for } s \leq \frac{m}{2} \text{ or } s \geq 2M.
\end{cases}
\end{equation*}

For $\xi\neq0$ define
\begin{equation*}\label{eq:Atilde}
 \widetilde A_{p,\gamma}(\xi)
 :=
 \textbf{I}\mathrm{d_{n}} +\chi(|\xi|)
 \left[
 |\xi|^\gamma
 \left(
 \textbf{I}\mathrm{d_{n}} + (p-2)\widehat\xi\otimes\widehat\xi
 \right) - \textbf{I}\mathrm{d_{n}}
 \right],
\end{equation*}
and set $\widetilde A_{p,\gamma}(0): =\textbf{I}\mathrm{d_{n}} $.
Because $\chi$ vanishes in a neighbourhood of the origin,
$\widetilde A_{p,\gamma}$ is continuous on all of $\mathbb R^n$.
Moreover,
\begin{equation}\label{eq:Atilde-close}
 \sup_{\xi\in\mathbb R^n}
 \left\|
 \widetilde A_{p,\gamma}(\xi)- \textbf{I}\mathrm{d_{n}}
 \right\|
 \leq
 \mathrm{C}(m,M)\bigl(|p-2|+|\gamma|\bigr).
\end{equation}
Hence, after decreasing $\varepsilon_s$,
\begin{equation*}\label{eq:Atilde-ell}
 \frac12 \textbf{I}\mathrm{d_{n}}
 \leq
 \widetilde A_{p,\gamma}(\xi)
 \leq 2 \textbf{I}\mathrm{d_{n}}
 \quad  \hbox{for every } \xi \in  \mathbb R^n.
\end{equation*}

By \eqref{eq:test-annulus-smooth}, every test gradient which is relevant
for the viscosity definition of $v$ lies in the region where $\chi=1$.
Therefore $v$ is also a viscosity solution of an uniformly
parabolic equation
\begin{equation*}\label{eq:extended-equation}
 v_t-
 \operatorname{tr}
 \left(
 \widetilde A_{p,\gamma}(Dv)D^2v
 \right)
 = f
 \quad   \hbox{in }   Q_1.
\end{equation*}
This elementary extension is the point which removes the singularity at
$Dv=0$ from the smooth regime.

\medskip
\noindent
\textit{Step 2}.
We now make the following claim:

\medskip
\noindent

{\it Claim.} There exist
\[
 0<\rho=\rho(n,\beta)<\frac14,
 \qquad
 \delta_s=\delta_s(n,\beta)>0
\]
such that, whenever $A:\mathbb R^n\to\mathbb S^n$ is continuous and
\begin{equation}\label{eq:A-close-I}
 \frac12  \textbf{I}\mathrm{d_{n}}  \leq A(\xi)\leq   2 \textbf{I}\mathrm{d_{n}},
 \qquad
 \sup_{\xi\in\mathbb R^n}\|A(\xi)- \textbf{I}\mathrm{d_{n}} \|\leq\delta_s,
\end{equation}
and $ \Phi $ solves
\[
 \Phi_t-\operatorname{tr}\bigl(A(D\Phi)D^2\Phi\bigr)=h
 \quad   \hbox{in }  Q_1,
\]
with
\[
 \|\Phi\|_{L^\infty(Q_1)}\leq1,
 \qquad
 \|h\|_{L^\infty(Q_1)}\leq\delta_s,
\]
there is an affine function
\[
 \ell(x):=a+b\cdot x
\]
such that
\begin{equation*}\label{eq:one-step-smooth}
 \sup_{Q_\rho}|\Phi -\ell|
 \leq\rho^{1+\beta},
 \qquad
 |a|+|b|\leq \mathrm{C}_n.
\end{equation*}

\medskip
\noindent

We prove this claim by contradiction.  If it is false, then there exist the sequence
$ (A_j)_{j}, (\Phi_j)_{j}, (h_j)_{j} $ such that

(i). \begin{equation*}
 \frac12  \textbf{I}\mathrm{d_{n}}  \leq A_{j}(\xi)\leq   2 \textbf{I}\mathrm{d_{n}},
 \qquad
 \sup_{\xi\in\mathbb R^n}\|A_{j}(\xi)- \textbf{I}\mathrm{d_{n}} \|\leq  1/j;
\end{equation*}

(ii). $ (\Phi_{j})_{j} $ satisfies, in the viscosity sense, the equation
\begin{equation*}
  (\Phi_{j})_t - \operatorname{tr}\bigl(A_{j}(D\Phi_{j})D^2\Phi_{j}\bigr)=h_{j}
 \quad   \hbox{in }  Q_1,
\end{equation*}
with
\[
 \|\Phi_{j}\|_{L^\infty(Q_1)} \leq  1,
 \quad  \text{and}  \quad
 \|h_j\|_\infty\longrightarrow  0;
\]
(iii). However, for all affine function $ \ell(x) $, there holds
\begin{equation}\label{all-affine}
  \sup_{Q_\rho}|\Phi_{j} -\ell|
  > \rho^{1+\beta},
 \quad   \ell(x):=a+b\cdot x,   \quad
 |a|+|b|\leq \mathrm{C}_n.
\end{equation}

The uniform ellipticity in \eqref{eq:A-close-I}, together with the parabolic Krylov--Safonov estimate (see, e.g., \cite{W92}), makes $ (\Phi_j)_{j} $ locally equicontinuous. Up to a
subsequence, we obtain
\[
 \Phi_j\longrightarrow \Phi_\infty
 \quad \hbox{locally uniformly in }Q_1.
\]
By the stability theorem for viscosity solutions,
\[
 (\Phi_\infty)_t-\Delta \Phi_\infty=0
 \quad   \hbox{in }Q_1.
\]
The interior estimates for the heat equation imply
\[
 \|\Phi_\infty\|_{C^{2,1}(Q_{1/2})}\leq \mathrm{C}_n.
\]
Let
\[
 \ell_\infty(x)
 :=
 \Phi_\infty(0,0)+D\Phi_\infty(0,0)\cdot x.
\]
Then
\[
 \sup_{Q_\rho}|\Phi_\infty-\ell_\infty|
 \leq \mathrm{C}_n\rho^2.
\]
Choose $\rho$ so small that
\[
 \mathrm{C}_n\rho^2\leq\frac14\rho^{1+\beta}.
\]
For $j$ sufficiently large, the local uniform convergence gives
\[
 \sup_{Q_\rho}|\Phi_j-\ell_\infty|
 \leq  \frac12 \rho^{1+\beta},
\]
contradicting the assumption \eqref{all-affine}. Hence this proves this claim.

\medskip
\noindent
\textit{Step 3}.
By dividing $v$ by a constant $ 1+\|v\|_{L^\infty(Q_1)} $, it is enough first to consider the normalized case
\[
 \|v\|_{L^\infty(Q_1)}\leq1.
\]
Choose $\varepsilon_s$ sufficiently small so that \eqref{eq:Atilde-close} is bounded
by $\delta_s$ and $\|f\|_\infty\leq\delta_s$.
We want to show that there exists a sequence of affine functions
$ \ell_j(x): =a_j+b_j\cdot x $ such that
\begin{equation}\label{eq:smooth-iteration}
 \sup_{Q_{\rho^j}}|v-\ell_j|
 \leq
 \rho^{j(1+\beta)},
\end{equation}
and
\begin{equation}\label{eq:smooth-increments}
 |a_{j+1}-a_j|
 \leq \mathrm{C} \rho^{j(1+\beta)},
 \qquad
 |b_{j+1}-b_j|
 \leq \mathrm{C}  \rho^{j\beta}.
\end{equation}

\medskip
\noindent

For this purpose, we proceed by induction. The case $ j=0 $ is trivial.  Assuming \eqref{eq:smooth-iteration} holds at level $j=k$, we now define
\begin{equation*}\label{eq:zk}
 \Phi_k(x,t)
 :=
 \frac{
 v(\rho^kx,\rho^{2k}t)
 -
 \ell_k(\rho^kx)
 }{
 \rho^{k(1+\beta)}
 }.
\end{equation*}
Then $\|\Phi_k\|_\infty\leq1$ and $(\Phi_k)_{k}$ satisfies
\begin{equation*}\label{eq:zk-equation}
 (\Phi_k)_t
 -
 \operatorname{tr}
 \left(
 A_k(D\Phi_k)D^2\Phi_k
 \right)
 =
 f_k
 \quad    \hbox{in }   Q_1,
\end{equation*}
where
\[
 A_k(\xi)
 :=
 \widetilde A_{p,\gamma}
 \bigl(
 b_k+\rho^{k\beta}\xi
 \bigr)
\]
and
\[
 f_k(x,t)
 :=
 \rho^{k(1-\beta)}
 f(\rho^kx,\rho^{2k}t).
\]
The crucial observation is that {\it Step~1} makes
the perturbation estimate invariant under this affine-gradient
translation:
\[
 \sup_{\xi\in\mathbb R^n}\|A_k(\xi)-\textbf{I}\mathrm{d_{n}}\|
 \leq\delta_s
\]
for every $k$.  Also, since $\beta<1$, then it yields
\[
 \|f_k\|_\infty
 \leq
 \|f\|_\infty
 \leq   \delta_s.
\]
Therefore, the claim in {\it Step~2} gives
\[
 \widetilde\ell_k(x):=\widetilde a_k+\widetilde b_k\cdot x
\]
with
\[
 \sup_{Q_\rho}|\Phi_k-\widetilde\ell_k|
 \leq\rho^{1+\beta},
 \qquad
 |\widetilde a_k|+|\widetilde b_k|
 \leq \mathrm{C}_n.
\]
Set
\[
 \ell_{k+1}(x)
 :=
 \ell_k(x)
 +
 \rho^{k(1+\beta)}
 \widetilde\ell_k(\rho^{-k}x).
\]
Then \eqref{eq:smooth-iteration} and \eqref{eq:smooth-increments} follow at level $ j=k+1$. As a consequence, the series in \eqref{eq:smooth-increments} shows that
$b_j \to b_\infty$ and
\[
 |b_j-b_\infty|\leq \mathrm{C} \rho^{j\beta}.
\]
As usual, if $\rho^{j+1}< r \leq \rho^j $, then
\begin{equation*}\label{eq:campanato-point}
 \sup_{Q_r}
 \left|
 v(x)-v(0,0) - b_\infty\cdot x
 \right|
 \leq
 \mathrm{C} r^{1+\beta}.
\end{equation*}
Repeating the construction after translating the centre to an arbitrary point of $Q_{1/2}$ and using overlapping interior cylinders gives
\[
 |Dv(x,t)-Dv(y,s)|
 \leq
 \mathrm{C}  \left(
 |x-y|^\beta+|t-s|^{\beta/2}
 \right),
\]
and the affine approximation also yields
\[
 |v(x,t)-v(x,s)|
 \leq
 \mathrm{C} |t-s|^{(1+\beta)/2}.
\]
This is to say,
\[
 v\in C^{1+\beta,\frac{1+\beta}{2}}(Q_{1/2}).
\]

For a non-normalized solution, divide by
\[
 K:=1+\|v\|_{L^\infty(Q_1)}.
\]
The normalized function satisfies an equation with coefficient
$
 \widehat A(\xi)
 :=
 \widetilde A_{p,\gamma}(K\xi),
$,
which still obeys the same global estimate
\[
 \sup_{\xi \in \mathbb{R}^{n}}  \|\widehat A(\xi)- \textbf{I}\mathrm{d_{n}} \|
 \leq\delta_s.
\]
Scaling back gives \eqref{eq:smooth-est}. Therefore, this completes the proof of this lemma.
\end{proof}

\subsection{Application at the stopping scale}
\label{subsec:application-stop}

Recall from Section~\ref{subsec:first-stop} that
\[
 M_*=\rho^{-\beta}(K_*+ \mathrm{C}_0).
\]
Let $\eta_*$ and $\varepsilon_*$ be the constants in
Lemma~\ref{lem:noncriticalLip} with $M=M_*$. Let $\varepsilon_s$ be
the threshold in the fixed-radius version of
Lemma~\ref{lem:smoothimproved}, applied on $Q_{3/4}$ with conclusion
on $Q_{1/2}$ and with $m=1$ and $M=M_*+1$.

Choose
\[
 0<\eta\leq\min\{1,\eta_*/2\}.
\]
For this fixed $\eta$, denote by $\varepsilon_{\mathrm{flat}}(\eta)$
the joint smallness threshold provided by Lemma~\ref{lem:onestep}
with $M=K_*$. We then choose
\begin{equation*}\label{eq:smallness-for-stop}
 0<\varepsilon_0\leq
 \min\left\{
 \varepsilon_{\mathrm{flat}}(\eta),
 \varepsilon_*,\varepsilon_s,\eta_*/2
 \right\},
\end{equation*}
decreasing it further if necessary to meet the structural conditions fixed in Section~\ref{Section 4}.

\begin{proposition}
\label{prop:regularity-at-stop}
Let $k_0$ be the first stopping index introduced in
Section~\ref{subsec:first-stop}, and set
\[
 v_{k_0}(x,t):=w_{k_0}(x,t)+q_{k_0}\cdot x.
\]
Then
\begin{equation}\label{eq:regularity-at-stop}
 \norm{v_{k_0}}_{C^{1+\beta,(1+\beta)/2}(Q_{1/2})}\leq \mathrm{C},
\end{equation}
where $ \mathrm{C} $ depends only on $n,\beta$ and the fixed stopping parameters,
and is independent of $k_0$. In particular, there is a vector $b_{k_0}$
such that
\begin{equation}\label{eq:expansion-at-stop}
 \left|v_{k_0}(x,t)-v_{k_0}(0,0)-b_{k_0}\cdot x\right|
 \leq \mathrm{C}  \left(|x|^{1+\beta}+|t|^{(1+\beta)/2}\right)
 \quad   \text{in }   Q_{1/2}.
\end{equation}
Moreover,
\begin{equation*}\label{eq:tangent-at-stop}
 b_{k_0}:=Dv_{k_0}(0,0),
 \quad
 |b_{k_0}-q_{k_0}|\leq1,
 \quad   \text{and}  \quad
 |b_{k_0}|\leq M_*+1.
\end{equation*}
\end{proposition}

\begin{proof}
By \eqref{eq:slope-at-stop} and \eqref{eq:bounds-at-stop},
\[
 2<|q_{k_0}|\leq M_*,
 \qquad
 \norm{w_{k_0}}_{L^\infty(Q_1)}
 +\norm{f_{k_0}}_{L^\infty(Q_1)}
 \leq\eta+\varepsilon_0\leq\eta_*.
\]
Lemma~\ref{lem:noncriticalLip} therefore gives
\begin{equation}\label{eq:wk-unit-lip}
 \operatorname{Lip}_x(w_{k_0};Q_{3/4})\leq1.
\end{equation}
The function $v_{k_0}$ solves
\[
 (v_{k_0})_t-|Dv_{k_0}|^\gamma\Delta_p^{\mathrm{N}} v_{k_0}=f_{k_0}
 \quad   \text{in }   Q_1.
\]
Every test function touching $v_{k_0}$ from above or below in $Q_{3/4}$
satisfies
\begin{equation*}\label{eq:gradientannulus}
 1\leq|D\phi|\leq M_*+1.
\end{equation*}
Furthermore, the normalized assumptions and \eqref{eq:fk} imply
\[
 |p-2|+|\gamma|+\norm{f_{k_0}}_{L^\infty(Q_1)}
 \leq\varepsilon_0\leq\varepsilon_s.
\]
We may therefore apply the fixed-radius version of
Lemma~\ref{lem:smoothimproved}. Since
\[
 \norm{v_{k_0}}_{L^\infty(Q_1)}\leq\eta+M_*\leq1+M_*,
\]
we obtain \eqref{eq:regularity-at-stop} with a constant independent of
$k_0$. Estimate \eqref{eq:expansion-at-stop} follows with
$b_{k_0}:=Dv_{k_0}(0,0)$. Finally, the spatial Lipschitz bound for
$w_{k_0}$ gives $|b_{k_0}-q_{k_0}|\leq1$, and hence
\eqref{eq:tangent-at-stop} is also achieved.
\end{proof}

\vspace{2mm}

\section{Pointwise estimates after the alternative}\label{Section 6}

The sign of $\gamma$ determines the geometry that should be used after the
alternative.  On the degenerate case $ \gamma \geq 0 $, the final estimate will be proved directly on standard parabolic cylinders. On the singular case, $ \gamma < 0 $ the natural cylinders remain intrinsic. This distinction is
important in the stopping alternative: a standard estimate for the rescaled
smooth solution does not, when $\gamma>0$, automatically control the longer
intrinsic cylinder.

For $z_0: =(x_0,t_0)$ and a vector $\xi_{z_0}\in\mathbb R^n$, set
\[
 P_{z_0}(x):=u(z_0)+\xi_{z_0}\cdot(x-x_0)
\]
and define
\[
 \mathcal C_r(z_0):=
 \begin{cases}
  Q_r(z_0),& \gamma\geq0,\\[1mm]
  Q_r^{r^\beta}(z_0),& \gamma<0.
 \end{cases}
\]
Thus the time length of $\mathcal C_r$ is $r^2$ for $\gamma\geq0$ and
$r^{2-\beta\gamma}$ for $\gamma<0$.

Combining the critical and stopping situations developed in Section~\ref{Section 5}, we obtain the following unified pointwise estimate.

\begin{proposition}
\label{prop:pointwise-correct}
Assume the normalized hypotheses \eqref{eq:normalized} and
\[
 |p-2|+|\gamma|\leq\varepsilon
\]
with $\varepsilon$ smaller than the constants fixed above.  Then, at every
interior base point $z_0$, after a fixed preliminary rescaling depending only
on its distance from the parabolic boundary, there exists a vector
$\xi_{z_0}$ such that
\begin{equation}\label{eq:pointwise-correct}
 \sup_{\mathcal C_r(z_0)}|u-P_{z_0}|
 \leq \mathrm{C} r^{1+\beta}
\end{equation}
for every sufficiently small $r$.  The constant $ \mathrm{C} $ is independent of the
scale at which the critical iteration stops.  Moreover
\[
 \xi_{z_0}=Du(z_0).
\]
In particular, we have
\begin{align}
 &
 \sup_{Q_r(z_0)}|u-P_{z_0}|\leq \mathrm{C} r^{1+\beta},  \quad  \text{as} \quad  \gamma\geq 0;
 \label{eq:point-standard-positive}    \\
 &
 \sup_{Q_r^{r^\beta}(z_0)}|u-P_{z_0}|\leq \mathrm{C} r^{1+\beta}, \quad  \text{as}  \quad   \gamma< 0.
 \label{eq:point-intrinsic-negative}
\end{align}
\end{proposition}

\begin{proof}
We first consider the particular case $ z_0 = 0 $. Now we treat the following two cases separately.

\medskip
\noindent

\textit{Case 1: the critical iteration never stops.}
At the discrete scales $r_j=\rho^j$, Section~\ref{Section 5.1} gives
\[
 \sup_{Q_{r_j}^{\lambda_j}}|u-u(0,0)|
 \leq \mathrm{C} r_j^{1+\beta},
 \quad \lambda_j=r_j^\beta,
\]
and the limiting slope is zero.  If $\gamma\geq0$, then it is obvious to see that $ r_j^{2-\beta\gamma}\geq r_j^2 $, hence $ Q_{r_j}\subset Q_{r_j}^{\lambda_j} $.

If $\gamma<0$, the cylinder $Q_{r_j}^{\lambda_j}$ is exactly the intrinsic
cylinder required in \eqref{eq:point-intrinsic-negative}.  Given an arbitrary
$r$ with $r_{j+1}<r\leq r_j$, the relevant cylinder of radius $r$ is
contained in the corresponding cylinder of radius $r_j$, while
$r_j\leq\rho^{-1}r$.  Therefore \eqref{eq:pointwise-correct} follows for all small $r$ in {\it Case 1}.

\medskip
\noindent

\textit{Case 2: the critical iteration stops for the first time at level
$k_0$.}
Recall that $ \tau_{k_0}:=r_{k_0}^2\lambda_{k_0}^{-\gamma}
 =r_{k_0}^{2-\beta\gamma} $
and
\[
 v_{k_0}(x,t)
 : =\frac{u(r_{k_0}x,\tau_{k_0}t)-a_{k_0}}{r_{k_0}\lambda_{k_0}}.
\]
By Proposition~\ref{prop:regularity-at-stop},
\begin{equation}\label{eq:vk-smooth-point}
 |v_{k_0}(x,t)-v_{k_0}(0,0)-b_{k_0}\cdot x|
 \leq \mathrm{C}  \bigl(|x|^{1+\beta}+|t|^{(1+\beta)/2}\bigr)
 \quad   \text{in }    Q_{1/2},
\end{equation}
where
\[
 b_{k_0}:=Dv_{k_0}(0,0),
 \quad   |b_{k_0}| \leq  M_*+1.
\]
Define the final slope by
\begin{equation*}\label{eq:final-slope-stop}
 \xi:=\lambda_{k_0}b_{k_0}.
\end{equation*}
Since
\begin{equation}\label{eq:scale-back-final}
 u(x,t)-u(0,0)-\xi\cdot x
 =r_{k_0}\lambda_{k_0}
 \left[
 v_{k_0}\!\left(\frac x{r_{k_0}},\frac t{\tau_{k_0}}\right)
 -v_{k_0}(0,0)-b_{k_0}\cdot\frac x{r_{k_0}}
 \right],
\end{equation}
we can treat the scales below $r_{k_0}$ directly:

{\it Case~2.1}. Assume first that $\gamma\geq0$ and let $0<R\leq r_{k_0}/2$.  Put
$s:=R/r_{k_0}$.  If $(x,t)  \in  Q_R $, then
\[
 \frac{|t|}{\tau_{k_0}}
 \leq\frac{R^2}{r_{k_0}^{2-\beta\gamma}}
 =s^2r_{k_0}^{\beta\gamma}
 \leq s^2.
\]
Thus $Q_R$ is mapped by the stopping-scale change of variables into $Q_s$. We combine \eqref{eq:vk-smooth-point} and \eqref{eq:scale-back-final} to obtain
\begin{equation}\label{eq:stop-positive-below}
 \sup_{Q_R}|u-u(0,0)-\xi\cdot x|
 \leq \mathrm{C} r_{k_0}\lambda_{k_0}s^{1+\beta}
 =\mathrm{C} R^{1+\beta}.
\end{equation}
Notice that only a standard cylinder is used here; no control of the longer
intrinsic cylinder is needed.

The temporal scaling makes the same point explicitly.  The smooth estimate
for $v_{k_0}$ gives
\[
 |Dv_{k_0}(0,t)-Dv_{k_0}(0,0)|\leq \mathrm{C} |t|^{\beta/2}
\]
and
\[
 |v_{k_0}(0,t)-v_{k_0}(0,0)|\leq \mathrm{C}  |t|^{(1+\beta)/2}.
\]
Consequently, for $|t|\leq\tau_{k_0}/4$,
\begin{align}
 |Du(0,t)-Du(0,0)|
 &\leq
 \mathrm{C}  r_{k_0}^{\beta^2\gamma/2}|t|^{\beta/2},
 \label{eq:positive-prefactor-gradient}\\
 |u(0,t)-u(0,0)|
 & \leq
 \mathrm{C}  r_{k_0}^{\beta\gamma(1+\beta)/2}
 |t|^{(1+\beta)/2}.
 \label{eq:positive-prefactor-u}
\end{align}
Both powers of $r_{k_0}$ are nonnegative when $\gamma\geq0$; since $r_{k_0}\leq1$,
the prefactors are bounded by one.  This is the precise mechanism that
recovers the standard parabolic time exponents in the stopping branch.

{\it Case~2.2}. Assume $\gamma<0$ and let $0<R\leq r_{k_0}/2$.  Again put $s=R/r_{k_0}$.
If $(x,t)\in Q_R^{R^\beta}$, then
\[
 \frac{|t|}{\tau_{k_0}}
 \leq
 \frac{R^{2-\beta\gamma}}{r_{k_0}^{2-\beta\gamma}}
 =s^{2-\beta\gamma}
 \leq s^2,
\]
because $2-\beta\gamma>2$ and $0<s<1$.  Thus the intrinsic cylinder at
scale $R$ is mapped into the standard cylinder $Q_s$ for $v_{k_0}$, and the
same calculation yields
\begin{equation}\label{eq:stop-negative-below}
 \sup_{Q_R^{R^\beta}}
 |u-u(0,0)-\xi\cdot x|
 \leq \mathrm{C}  R^{1+\beta}.
\end{equation}

It remains to connect the final tangent plane to the larger scales already
constructed before stopping.  Since $q_{k_0}= \frac{ l_{k_0}}{\lambda_{k_0}}$, the slope bound
\eqref{eq:tangent-at-stop} gives
\begin{equation}\label{eq:xi-lk}
 |\xi-l_{k_0}|
 =\lambda_{k_0}|b_{k_0}-q_{k_0}|
 \leq\lambda_{k_0}.
\end{equation}
Moreover, summing the slope increments in \eqref{eq:increments}, for every
$j\leq k_0$,
\[
 |l_{k_0}-l_j|
 \leq \mathrm{C}  \sum_{m=j}^{k_0-1}\lambda_m
 \leq \mathrm{C}  \lambda_j.
\]
Hence
\begin{equation}\label{eq:xi-lj}
 |\xi-l_j|\leq \mathrm{C}  \lambda_j.
\end{equation}
At the origin the flatness estimate also gives
\[
 |u(0,0)-a_j|\leq  \mathrm{C}   r_j\lambda_j.
\]
Therefore, with $L_j(x):= a_j+l_j\cdot x$,
\begin{equation}\label{eq:final-plane-vs-Lj}
 \sup_{B_{r_j}}
 |P_0-L_j|
 \leq \mathrm{C}  r_j\lambda_j.
\end{equation}
Combining \eqref{eq:final-plane-vs-Lj} with \eqref{eq:flatnessk}, we obtain
\begin{equation}\label{eq:large-scale-final-plane}
 \sup_{Q_{r_j}^{\lambda_j}}
 |u-P_0|
 \leq \mathrm{C}  r_j^{1+\beta},
 \quad   j \leq k_0.
\end{equation}
If $\gamma\geq0$, then $Q_R\subset Q_{r_j}^{\lambda_j}$ whenever
$r_{j+1}<R\leq r_j$; if $\gamma<0$, then
$Q_R^{R^\beta}\subset Q_{r_j}^{\lambda_j}$ for the same range of $R$.
Since $r_j\leq\rho^{-1}R$, \eqref{eq:large-scale-final-plane} supplies the
pointwise estimate above the stopping scale. For
$r_{k_0}/2<R\leq r_{k_0}$, use the same estimate with $j=k_0$
and $r_{k_0}<2R$. Together with
\eqref{eq:stop-positive-below} or \eqref{eq:stop-negative-below}, this proves
\eqref{eq:pointwise-correct} at every scale.  The estimate itself implies
that $u$ is differentiable at the base point with spatial gradient $\xi$.

Finally, for an arbitrary point $z_0\in Q_{1/2}$. We choose once and for all a small number $ r_0=r_0(\gamma_*)>0 $ such that
\[
 B_{r_0}(x_0)\times
 (t_0-r_0^{2+\gamma},t_0]
 \subset Q_1
\]
for every $z_0\in Q_{1/2}$ and every $|\gamma|\leq\gamma_*$.  Applying the preceding argument for $ z_0 = 0 $ to
\[
 \widetilde u(x,t)
 :=u(x_0+r_0x,t_0+r_0^{2+\gamma}t)-u(z_0).
\]
Because the normalization is stated in terms of oscillation,
$\osc\widetilde u\leq\eta$, and the rescaled right-hand side is no larger than the original one.  Scaling back changes only the universal constant. This finishes the proof.
\end{proof}

The preceding pointwise estimate yields the desired spatial and temporal H\"older moduli by a Campanato comparison of tangent planes. To be more precise,

\begin{proposition}\label{prop:campanato-consequences}
Under the hypotheses of Proposition~\ref{prop:pointwise-correct}.
Let
\[
 \vartheta_\gamma:=
 \begin{cases}
  2,  &   \quad   \text{if}  \quad  \gamma  \geq   0,   \\
  2-\beta\gamma,      &  \quad   \text{if}  \quad  \gamma  <  0,
 \end{cases}
\]
then on every compact subcylinder $ Q' \Subset Q_{1}$, for $ (x,t), (y,t), (x,s) \in Q' $, we have the following estimates:
\begin{equation}\label{eq:spacegrad}
 |Du(x,t)-Du(y,t)|\leq \mathrm{C} |x-y|^\beta,
\end{equation}
\begin{equation}\label{eq:timegrad-general}
 |Du(x,t)-Du(x,s)|
 \leq \mathrm{C} |t-s|^{\frac{\beta}{\vartheta_\gamma}},
\end{equation}
and
\begin{equation}\label{eq:utime-general}
 |u(x,t)-u(x,s)|
 \leq \mathrm{C} |t-s|^{\frac{1+\beta}{\vartheta_\gamma}}.
\end{equation}
\end{proposition}

\begin{proof}
We first record the following elementary claim that will be used repeatedly.

\medskip
\noindent

{\it Claim.} Let
\[
    L_1(X)=a_1+b_1\cdot X,
    \quad
    L_2(X)=a_2+b_2\cdot X,
\]
and assume that, on a ball $B_{\mathrm{c}r}(X_*)$ with a fixed $ \mathrm{c}  \in (0,1) $,
\begin{equation}\label{eq:affine-bound}
    |L_1(X)-L_2(X)|\le A r^{1+\beta}.
\end{equation}
Then
\begin{equation}\label{eq:slope-comparison}
    |b_1-b_2|\le \mathrm{C}(\mathrm{c})A r^\beta.
\end{equation}
Indeed, let $ \vec{e} \in \mathbb S^{n-1}$ and choose
\[
    X_\pm :=X_*\pm \frac{\mathrm{c}r}{2}\vec{e}.
\]
Noting that both points $  X_\pm$ belong to $B_{\mathrm{c}r}(X_*)$, then subtracting the two inequalities obtained from \eqref{eq:affine-bound} at $X_+$ and $X_-$ gives
\[
    \mathrm{c}r\, |(b_1-b_2)\cdot \vec{e}|
    \le 2A r^{1+\beta}.
\]
Hence
\[
    |(b_1-b_2)\cdot \vec{e}|
    \le \frac{2A}{\mathrm{c}}r^\beta.
\]
Taking the supremum over $ \vec{e} \in \mathbb S^{n-1}$ proves \eqref{eq:slope-comparison}.

\medskip
\noindent

{\it Step 1: Proof of~\eqref{eq:spacegrad}.}
Fix a time $t$ and two spatial points $x,y$ in the compact subcylinder $ Q'$. Set
\[
    d:=|x-y|.
\]
We first assume that $d$ is sufficiently small so that all cylinders used below stay inside the region where Proposition~\ref{prop:pointwise-correct} is available.
Let
\[
    z_1:=(x,t),
    \quad
    z_2:=(y,t),
\]
and denote the corresponding tangent planes by
\begin{align*}
P_1(X)&:=u(x,t)+Du(x,t)\cdot(X-x),\\
P_2(X)&:=u(y,t)+Du(y,t)\cdot(X-y).
\end{align*}
At the common time level $t$, for $X \in B_r(x) \cap  B_r(y) $, Proposition~\ref{prop:pointwise-correct} gives
\begin{equation}\label{eq:P1-u}
    |u(X,t)-P_1(X)|\le \mathrm{C}_0r^{1+\beta}
\end{equation}
and
\begin{equation}\label{eq:P2-u}
    |u(X,t)-P_2(X)|\le \mathrm{C}_0r^{1+\beta}.
\end{equation}

Let
\[
    x_*:=\frac{x+y}{2}.
\]
Since $r=4d$, the ball $B_{r/8}(x_*)=B_{d/2}(x_*)$ is contained in both $B_r(x)$ and $B_r(y)$. Therefore, on $B_{r/8}(x_*)$, combining \eqref{eq:P1-u} and \eqref{eq:P2-u} yields
\[
    |P_1(X)-P_2(X)|
    \le 2\mathrm{C}_0r^{1+\beta}.
\]
Applying the preceding claim, with $ \mathrm{c}=1/8$, we obtain
\[
    |Du(x,t)-Du(y,t)|
    \le \mathrm{C}r^\beta
    \le \mathrm{C}|x-y|^\beta.
\]

\medskip
\noindent

{\it Step 2: Proof of \eqref{eq:timegrad-general}.}
Fix $x$ and let $s<t$. Set $ h:= t-s $ and choose
\begin{equation}\label{eq:r-time}
    r:=2h^{1/\vartheta_\gamma}.
\end{equation}
For $h$ sufficiently small, the cylinders below remain in the interior. Moreover,
\[
    r^{\vartheta_\gamma}
    =2^{\vartheta_\gamma}h
    >h=t-s.
\]

Let
\begin{align*}
P_t(X)&:=u(x,t)+Du(x,t)\cdot(X-x),\\
P_s(X)&:=u(x,s)+Du(x,s)\cdot(X-x).
\end{align*}
For every $X\in B_{r/2}(x)$, the point $(X,s)$ lies in $\mathcal C_r(x,t)$, and therefore Proposition~\ref{prop:pointwise-correct} gives
\begin{equation}\label{eq:later-plane}
    |u(X,s)-P_t(X)|\le \mathrm{C}_0r^{1+\beta}.
\end{equation}
Since $s$ is the top time-level of $\mathcal C_r(x,s)$, the same pointwise estimate at the earlier base point gives
\begin{equation}\label{eq:earlier-plane}
    |u(X,s)-P_s(X)|\le \mathrm{C}_0r^{1+\beta}.
\end{equation}
Combining \eqref{eq:later-plane} and \eqref{eq:earlier-plane}, we get
\[
    |P_t(X)-P_s(X)|\le 2\mathrm{C}_0r^{1+\beta}
    \quad   \text{for }   X\in B_{r/2}(x).
\]
Now we use the preceding claim, with $ \mathrm{c} = 1/2 $, yields
\[
    |Du(x,t)-Du(x,s)|\le \mathrm{C} r^\beta.
\]
Using \eqref{eq:r-time} and $ r^\beta
    =2^\beta h^{\frac{\beta}{\vartheta_\gamma}} $, we obatin
\[
    |Du(x,t)-Du(x,s)|
    \le \mathrm{C} |t-s|^{\frac{\beta}{\vartheta_\gamma}}.
\]

\medskip
\noindent

{\it Step 3: Proof of \eqref{eq:utime-general}.}
We keep the same choice \eqref{eq:r-time}. Since $(x,s)\in \mathcal C_r(x,t) $, applying Proposition~\ref{prop:pointwise-correct} at the point $(x,t)$ and evaluating at the same spatial point $X=x$ gives
\[
    |u(x,s)-P_t(x)|\le \mathrm{C}_0r^{1+\beta}.
\]
But $ P_t(x)=u(x,t) $ since the spatial linear term vanishes at $X=x$. Hence
\[
    |u(x,t)-u(x,s)|\le \mathrm{C}_0r^{1+\beta}.
\]
Substituting \eqref{eq:r-time}, we obtain
\[
    r^{1+\beta}
    =2^{1+\beta}|t-s|^{(1+\beta)/\vartheta_\gamma},
\]
and consequently
\[
    |u(x,t)-u(x,s)|
    \le \mathrm{C} |t-s|^{(1+\beta)/\vartheta_\gamma},
\]
which implies the desired \eqref{eq:utime-general}.
\end{proof}

\vspace{2mm}

\section{Proof of Theorem~\ref{thm:main}}\label{Section 7}

We are now in a position to combine the preceding estimates and complete the proof of Theorem~\ref{thm:main}.

\begin{proof}[{ Proof of Theorem~\ref{thm:main}}]
We use the constants fixed in the preceding sections: first
$\beta$ and $\gamma_*$, then $\rho$, the stopping annulus, $\eta$,
and finally the joint threshold $\varepsilon_0$. Choose
\[
 \varepsilon_\alpha
 \leq
 \min\left\{
 \frac{\varepsilon_0}{2},\gamma_*
 \right\}.
\]
Assume
\[
 |p-2|+|\gamma|\leq\varepsilon_\alpha.
\]
In the normalized argument, we further require
\[
 \|f\|_{L^\infty(Q_1)}
 \leq\frac{\varepsilon_0}{2}.
\]
Hence
\[
 |p-2|+|\gamma|+\|f\|_{L^\infty(Q_1)}
 \leq\varepsilon_0,
\]
so the smallness condition in \eqref{eq:normalized} is satisfied.
Moreover, this condition on the source is preserved under the
rescaling \eqref{eq:fk}.

We treat the following two cases separately.

\medskip
\noindent

{\bf Case 1.} If $\gamma\geq0$, then $\vartheta_\gamma=2$.  By Proposition~\ref{prop:campanato-consequences}, we have
\[
 |Du(x,t) - Du(y,s)|
 \leq \mathrm{C} \bigl(|x-y|^\beta+|t-s|^{\frac{\beta}{2}}\bigr)
\]
and
\[
 |u(x,t)-u(x,s)|
 \leq \mathrm{C}  |t-s|^{\frac{1+\beta}{2}}.
\]
Thus the normalized solution belongs to
$C_{\mathrm{loc}}^{1+\beta,\frac{1+\beta}{2}} $.  Since $\beta> \alpha $, this implies $ u\in C_{\mathrm{loc}}^{1+\alpha,\frac{1+\alpha}{2}} $.

\medskip
\noindent

{\bf Case 2.} If $\gamma<0$, then $\vartheta_\gamma=2-\beta\gamma$.  The spatial estimate
still has exponent $\beta$, while the temporal exponents are
\[
 \frac{\beta}{2-\beta\gamma}
 \quad\text{and}\quad
 \frac{1+\beta}{2-\beta\gamma}.
\]
Since $|\gamma|\leq\gamma_*$, the choice
\eqref{eq:timeconditions} gives
\[
 \frac{\beta}{2-\beta\gamma}
 \geq\frac{\beta}{2+\beta\gamma_*}
 \geq\frac{\alpha}{2}
\]
and
\[
 \frac{1+\beta}{2-\beta\gamma}
 \geq\frac{1+\beta}{2+\beta\gamma_*}
 \geq\frac{1+\alpha}{2}.
\]
Together with $\beta>\alpha$, this again gives $ u\in C_{\mathrm{loc}}^{1+\alpha,\frac{1+\alpha}{2}} $.

\medskip
\noindent

{\bf Conclusion.}
Fix $p$ and $\gamma$ such that $|p-2|+|\gamma|\leq\varepsilon_\alpha$.
By preceding {\bf Case~1} and {\bf Case~2}, we have that there exists $ \mathrm{C}_*= \mathrm{C}_*(n,\alpha)$ such that
\begin{equation}\label{eq:main-normalized-estimate}
  \|v\|_{C^{1+\alpha,(1+\alpha)/2}(Q_{1/2})}\leq \mathrm{C}_*
\end{equation}
whenever $v$ solves the equation in $Q_1$ with right-hand side $f$
and satisfies
\[
  v(0,0)=0,\qquad
  \operatorname{osc}_{Q_1}v\leq\eta,\qquad
  \|f\|_{L^\infty(Q_1)}\leq\frac{\varepsilon_0}{2}.
\]

Now let $u$ be the solution in Theorem~\ref{thm:main}. Set
\[
  U:=\|u\|_{L^\infty(Q_1)},\qquad
  F_0:=\|f\|_{L^\infty(Q_1)},\qquad
  A:=\max\left\{4,\frac{2U}{\eta},
                    \frac{2F_0}{\varepsilon_0}\right\},
\]
and choose
\[
  r:=A^{-1},\qquad
  \tau:=r^{2+\gamma}A^{-\gamma}=A^{-2(1+\gamma)}.
\]
Since $A\geq4$ and $|\gamma|\leq\gamma_*\leq1/4$, we have
\begin{equation}\label{eq:main-time-scale-bounds}
  r\leq\frac14,\qquad
  A^{-3}\leq\tau\leq A^{-1}\leq\frac14.
\end{equation}
Thus $B_r(x_0)\times(t_0-\tau,t_0]\subset Q_1$ for every
$z_0=(x_0,t_0)\in Q_{1/2}$. Define
\[
  v_{z_0}(x,t)
  :=\frac{u(x_0+rx,t_0+\tau t)-u(x_0,t_0)}{A},
  \qquad (x,t)\in Q_1.
\]
By the scaling in Section~\ref{Section2.2}, $v_{z_0}$ solves
\[
  (v_{z_0})_t-|Dv_{z_0}|^\gamma\Delta_p^N v_{z_0}
  =f_{z_0}(x,t)
  :=\frac{\tau}{A}f(x_0+rx,t_0+\tau t)
  \qquad\text{in }Q_1.
\]
Moreover, $v_{z_0}(0,0)=0$ and
\[
  \operatorname{osc}_{Q_1}v_{z_0}
  \leq\frac{2U}{A}\leq\eta,\qquad
  \|f_{z_0}\|_{L^\infty(Q_1)}
  \leq\frac{\tau F_0}{A}\leq\frac{\varepsilon_0}{2}.
\]
Hence \eqref{eq:main-normalized-estimate} applies to every $v_{z_0}$
with the same constant $ \mathrm{C}_*$. In particular, $u$ is spatially
differentiable at $z_0$, and
\begin{equation}\label{eq:main-restored-gradient}
  |Du(z_0)|=\frac{A}{r}|Dv_{z_0}(0,0)|
  \leq\frac{\mathrm{C}_*A}{r}
  \qquad\text{for every }z_0\in Q_{1/2}.
\end{equation}

Let $(x,t),(y,s)\in Q_{1/2}$, with $s\leq t$. If $|x-y|< \frac{r}{2}$
and $t-s<\frac{\tau}{4} $, take $z_0=(x,t)$. The point $(y,s)$ then
corresponds to
\[
  \left(\frac{y-x}{r},\frac{s-t}{\tau}\right)\in Q_{1/2}.
\]
Applying \eqref{eq:main-normalized-estimate} at this point and at
the origin gives
\begin{equation}\label{eq:main-gradient-near}
  |Du(x,t)-Du(y,s)|
  \leq\frac{\mathrm{C}_*A}{r}
  \left(
    \frac{|x-y|^\alpha}{r^\alpha}
    +\frac{|t-s|^{\alpha/2}}{\tau^{\alpha/2}}
  \right).
\end{equation}
At the same spatial point, \eqref{eq:main-normalized-estimate} gives
\begin{equation}\label{eq:main-time-near}
  |u(x,t)-u(x,s)|
  \leq \mathrm{C}_*A\frac{|t-s|^{(1+\alpha)/2}}{\tau^{(1+\alpha)/2}}
  \qquad\text{if }t-s<\frac{\tau}{4}.
\end{equation}

If $|x-y|\geq \frac{r}{2}$ or $t-s\geq \frac{\tau}{4} $, then
\eqref{eq:main-restored-gradient} yields
\begin{equation}\label{eq:main-gradient-far}
  \frac{|Du(x,t)-Du(y,s)|}{|x-y|^\alpha+|t-s|^{\alpha/2}}
  \leq
  \frac{2\mathrm{C}_*A/r}{\min\{(r/2)^\alpha,(\tau/4)^{\alpha/2}\}}.
\end{equation}
For $t-s\geq \frac{\tau}{4}$, the bound $|u|\leq U$ gives
\begin{equation}\label{eq:main-time-far}
  \frac{|u(x,t)-u(x,s)|}{|t-s|^{(1+\alpha)/2}}
  \leq 2U\left(\frac{4}{\tau}\right)^{(1+\alpha)/2}.
\end{equation}
Since $r^{-1}=A$ and
$\tau^{-1/2}\leq A^{3/2}$ by \eqref{eq:main-time-scale-bounds}, now we combine \eqref{eq:main-restored-gradient}--\eqref{eq:main-time-far} to obtain
\[
  \|u\|_{C^{1+\alpha,(1+\alpha)/2}(Q_{1/2})}
  \leq \mathrm{C}(n,\alpha, \|u\|_{L^\infty(Q_1)},\|f\|_{L^\infty(Q_1)}),
\]
uniformly for $|p-2|+|\gamma|\leq\varepsilon_\alpha$. This completes the proof of Theorem~\ref{thm:main}.
\end{proof}

\vspace{4mm}

\noindent {\bf Acknowledgments.} F. Jiang has been supported by the National Natural Science Foundation of China (No. 12271093) and the Jiangsu Provincial Scientific Research Center of Applied Mathematics (Grant No. BK20233002), and Shanghai Institute for Mathematics and Interdisciplinary Sciences (SIMIS) under grant number SIMIS-ID-2025-AD.

\vspace{3mm}

\noindent {\bf AI Declaration.} During the preparation of this manuscript, the authors used GPT-6 Astra solely for English language editing and polishing, including improving grammar, sentence structure, and linguistic clarity. The tool was not used to develop mathematical arguments, construct proofs, or generate mathematical results. The authors carefully reviewed all suggested revisions and take full responsibility for the accuracy and final content of the manuscript.


\vspace{3mm}


\begin{thebibliography}{GPPSVG17}
\bibliographystyle{alpha}







\bibitem[ABM23]{ABM23}
M. Akman, A. Banerjee and I. H. Munive, Borderline gradient continuity for the normalized $ p$-parabolic operator, \newblock {\em J. Geom. Anal.},  \textbf{33} (2023), Art. 264.



\bibitem[ASU26]{ASU26}
C.~Alcantara, M.~Santos and J.~M.~Urbano, Gradient regularity for a class of singular or degenerate elliptic equations, \newblock {\em J. Differential Equations.}, \textbf{458} (2026), 114040.




\bibitem[AS26]{AS26}
C.~Alcantara and M.~Santos, Regularity for the normalized $ p$-Laplacian equation with an arbitrary degeneracy law, \newblock {\em Adv. Calc. Var.}, \textbf{19} (2026), 271--288.





\bibitem[AS22]{AS22}
P.~Andrade and M.~Santos, Improved regularity for the parabolic normalized $p$-Laplace equation, \newblock {\em Calc. Var. Partial Differential Equations.}, \textbf{61} (2022), no. 5, Paper No. 196, 13 pp.



\bibitem[AP18]{AP18}
A.~Attouchi and M.~Parviainen, H\"older regularity for the gradient of the inhomogeneous parabolic normalized $p$-Laplacian, \newblock {\em Commun. Contemp. Math.}, \textbf{20} (2018), no. 4, 1750035.



\bibitem[A20]{A20}
A.~Attouchi, Local regularity for quasi-linear parabolic equations in non-divergence form, \newblock {\em Nonlinear Anal.}, \textbf{199} (2020), 112051.


\bibitem[AR18]{AR18}
A.~Attouchi and E.~Ruosteenoja, Remarks on regularity for $ p$-Laplacian type equations in non-divergence form. \newblock {\em J. Differential Equations.}, \textbf{265} (2018), no. 5, 1922--1961.



\bibitem[AR20]{AR20}
A.~Attouchi and E.~Ruosteenoja, Gradient regularity for a singular parabolic equation in non-divergence form, \newblock {\em Discrete Contin. Dyn. Syst.}, \textbf{40} (2020), no. 10, 5955--5972.



\bibitem[CC95]{CC95}
L.~Caffarelli and X.~Cabr\'{e}, {\em Fully nonlinear elliptic equations}, volume 43. \newblock American Mathematical Society, 1995.





\bibitem[CIL92]{CIL92}
M.~G.~Crandall, H.~Ishii and P.-L.~Lions, User's guide to viscosity solutions of second order partial differential equations, \emph{ Bull. Amer. Math. Soc. (N.S.)}, \textbf{27} (1992), 1--67.





\bibitem[FZ21]{FZ21}
Y.~Fang and C.~Zhang, Gradient H\"older regularity for parabolic normalized $p(x,t)$-Laplace equation, \newblock {\em J. Differential Equations.}, \textbf{295} (2021), 211--232.



\bibitem[FP23]{FP23}
Y.~Feng, M.~Parviainen and S.~Sarsa, A systematic approach on the second order regularity of solutions to the general parabolic $p$-Laplace equation, \newblock {\em Calc. Var. Partial Differential Equations.}, \textbf{62} (2023), Art. 204.





\bibitem[IJS19]{IJS19}
C.~Imbert, T.~Jin and L.~Silvestre, H\"older gradient estimates for a class of singular or degenerate parabolic equations, \newblock {\em  Adv. Nonlinear Anal.}, \textbf{8} (2019), 845--867.


 \bibitem[CL12]{CL12}
       C. Imbert, L. Silvestre, {\em Introduction to fully nonlinear parabolic equations}, An introduction to the K\"ahler-Ricci flow, 7--88, Lecture Notes in Math., 2086, Springer, Cham, 2013.


\bibitem[JS17]{JS17}
T.~Jin and L.~Silvestre, H\"older gradient estimates for parabolic homogeneous $p$-Laplacian equations, \newblock {\em J. Math. Pures Appl.}, \textbf{108} (2017), 63--87.




\bibitem[LLYZ25]{LLYZ25}
S.-C.~Lee, Y.~Lian, H.~Yun and K.~Zhang, Boundary H\"older gradient estimates for parabolic $p$-Laplace type equations, \newblock {\em arXiv:2506.01018.}, 2025.









\bibitem[PRS20]{PRS20}
E.~Pimentel, G.~Rampasso and M.~Santos, Improved regularity for the $ p$-Poisson equation, \newblock {\em Nonlinearity.}, \textbf{33} (2020), 3050--3061.




\bibitem[WYJ25]{WYJ25}
     J. Wang, Y. Yin and F. Jiang, Regularity of solutions to degenerate normalized $ p$-Laplacian equation with general variable exponents, \newblock {\em Potential Anal.}, \textbf{63} (2025), 1963--2000.







\bibitem[W92]{W92}
L. Wang, On the regularity theory of fully nonlinear parabolic equations. I, \newblock {\em Comm. Pure Appl. Math.}, \textbf{45} (1992), 27--76.




\bibitem[W26]{W26}
X.~Wang, H\"older estimates on the gradient of viscosity solutions of degenerate Bellman type equations involving $ p$-Laplacian type operators, \newblock {\em Nonlinear Anal.}, \textbf{271} (2026), Paper No. 114156, 16 pp.



%
%
%
%
%
%
%
%
%
%
%





\end{thebibliography}
\end{document}